\documentclass[notitlepage]{amsart}
\usepackage{amssymb,amsfonts,latexsym}
\usepackage{graphics,verbatim}
\usepackage{graphicx}

\newtheorem{theorem}{Theorem}[section]
\newtheorem{proposition}{Proposition}[section]
\newtheorem{lemma}{Lemma}[section]

\newtheorem{remark}{Remark}[section]
\newtheorem{definition}{Definition}[section]

\newcommand{\eps}{\varepsilon}

\newcommand{\To}{\longrightarrow}

\newcommand{\norm}[1]{\left\Vert#1\right\Vert}

\newcommand{\be} {\begin{equation}}
\newcommand{\ee} {\end{equation}}
\newcommand{\bea} {\begin{eqnarray}}
\newcommand{\eea} {\end{eqnarray}}
\newcommand{\Bea} {\begin{eqnarray*}}
\newcommand{\Eea} {\end{eqnarray*}}
\newcommand{\pa} {\partial}

\newcommand{\al} {\alpha}
\newcommand{\ba} {\beta}
\newcommand{\de} {\delta}
\newcommand{\na}{\nabla}
\newcommand{\ga} {\gamma}
\newcommand{\Ga} {\Gamma}
\newcommand{\Om} {\Omega}
\newcommand{\om} {\omega}
\newcommand{\De} {\Delta}
\newcommand{\la} {\lambda}

\newcommand{\no} {\nonumber}
\newcommand{\noi} {\noindent}
\newcommand{\lab} {\label}
\newcommand{\va} {\varphi}
\newcommand{\f}{\frac}
\newcommand{\R}{\mathbb R}
\newcommand{\N}{\mathbb N}
\newcommand{\Rn}{\mathbb R^N}

\catcode`\@=11

\makeatletter \@addtoreset{equation}{section} \makeatother

\begin{document}

\title[elliptic equation with supercritical exponents]{Semilinear nonlocal elliptic equations with critical and supercritical exponents}
\author{Mousomi Bhakta} 
\author{ Debangana Mukherjee}
\address{Department of Mathematics, Indian Institute of Science Education and Research, Dr. Homi Bhaba Road, Pune-411008, India. }

\email{ M. Bhakta: mousomi@iiserpune.ac.in} 
\email{ D. Mukherjee: debangana18@gmail.com}

\subjclass[2010]{Primary 35B08, 35B40, 35B44}
\keywords{super-critical exponent, fractional laplacian, Pohozaev identity, nonexistence, entire solution, decay estimate, gradient estimate, radial symmetry, critical exponent,  nonlocal.}
\date{}

\begin{abstract} We study the problem
\begin{equation*}
\left\{\begin{aligned}
      (-\Delta)^s u &=u^p - u^q \quad\text{in }\quad \mathbb{R}^N, \\
      u &\in \dot{H}^s(\mathbb{R}^N)\cap L^{q+1}(\mathbb{R}^N), \\
      u&>0 \quad\text{in}\quad\mathbb{R}^N,
          \end{aligned}
  \right.
\end{equation*}
where $s\in(0,1)$ is a fixed parameter, $(-\Delta)^s$ is the fractional Laplacian in $\mathbb{R}^N$, $q>p\geq \frac{N+2s}{N-2s}$ and $N>2s$.  For every $s\in(0,1)$, we establish regularity results of solutions of above equation (whenever solution exists) and we show that every solution is a classical solution. Next, we derive certain decay estimate of solutions  and the gradient of solutions at infinity for all $s\in(0,1)$.  Using those decay estimates, we  prove Pohozaev type identity in $\Rn$  and we show that the above problem does not have any solution when $p=\frac{N+2s}{N-2s}$.  We  also discuss radial symmetry and decreasing property of the solution and prove that when $p>\frac{N+2s}{N-2s}$, the above problem admits a solution . Moreover, if we consider the above equation in a bounded domain with Dirichlet boundary condition, we prove that it admits a solution for every $p\geq \frac{N+2s}{N-2s}$ and every solution is a classical solution.\end{abstract}

\maketitle

\section{ Introduction}

In this paper,  we consider the following problem:
\begin{equation}
  \label{entire}
\left\{\begin{aligned}
      (-\De)^s u &=u^p -u^q \quad\text{in }\quad \Rn, \\
      u &\in \dot{H}^s(\Rn)\cap L^{q+1}(\Rn), \\
      u&>0 \quad\text{in}\quad\Rn,
 \end{aligned}
  \right.
\end{equation}
and
\begin{equation}
  \label{eq:a3'}
\left\{\begin{aligned}
      (-\De)^s u &=u^p -u^q \quad\text{in }\quad \Om, \\
      u &=0 \quad\text{in}\quad\Rn\setminus\Om,\\
      u&>0 \quad\text{in}\quad\Om,\\
      u &\in H^s(\Om)\cap L^{q+1}(\Om), \\
       \end{aligned}
  \right.
\end{equation}
where  $s\in(0,1)$ is fixed, $(-\De)^s$ denotes the  fractional Laplace operator defined,  up to a normalization factors, as
\begin{align} \label{De-u}
  -\left(-\Delta\right)^s u(x)=\frac{c_{N,s}}{2}\int_{\mathbb{R}^N}\frac{u(x+y)-2u(x)+u(x-y)}{|y|^{N+2s}}dy, \quad x\in\Rn,
\end{align}
where $$c_{N,s}:=\f{2^{2s}s\Ga(\f{N}{2}+s)}{\pi^\f{N}{2}\Ga(1-s)}.$$
In \eqref{entire} and \eqref{eq:a3'}, $q>p\geq 2^*-1=\f{N+2s}{N-2s}$ and $N>2s$. In \eqref{eq:a3'}, $\Om$ is a bounded subset of $\Rn$ with smooth boundary.

We denote by $H^s(\Om)$ the usual fractional Sobolev space endowed with the so-called Gagliardo norm
\be\lab{norm-H}\norm{g}_{H^s(\Om)}=\norm{g}_{L^2(\Om)}+\bigg(\int_{\Om \times \Om} \frac{|g(x)-g(y)|^2}{|x-y|^{N+2s}}dxdy\bigg)^{1/2}.\ee
For further details on the fractional Sobolev spaces we refer to \cite{NPV} and  the references therein.  Note that in problem \eqref{eq:a3'},   Dirichlet boundary data is given in  $\Rn\setminus\Om$ and not simply
 on $\pa\Om$. Therefore, for the Dirichlet boundary value problem in the bounded domain,  we need to introduce a new functional space $X_0$, which, in our opinion, is the suitable space to work with.   
\be\label{eq:X0}
X_0:=\{v\in H^s(\Rn): v=0\quad \text{in}\quad \Rn\setminus\Om\}.
\ee
By \cite[Lemma 6 and 7]{SerVal3}, it follows that 
\be\label{norm-X}
||v||_{X_0}=\displaystyle\left(\int_{Q}\f{|v(x)-v(y)|^2}{|x-y|^{N+2s}}dxdy,\right)^\f{1}{2},
\ee
where $Q=\R^{2N}\setminus (\Om^c\times\Om^c)$,
is a norm on $X_0$ and $(X_0, ||.||_{X_0})$ is a Hilbert space, with the inner product 
$$<u, v>_{X_0}=\int_{Q}\f{(u(x)-u(y))(v(x)-v(y))}{|x-y|^{N+2s}}dxdy.$$
Observe that, norms in \eqref{norm-H} and \eqref{norm-X} are not same, since $\Om\times\Om$ is strictly contained in $Q$. Clearly, the integral in \eqref{norm-X} can be extended to whole $\mathbb{R}^{2N}$ as $v=0$ in $\Rn\setminus\Om$. It is well known that the embedding $X_0\hookrightarrow L^r(\Rn)$ is compact, for any $r\in[1, 2^*)$ (see \cite[Lemma 8]{SerVal3}) and  $X_0\hookrightarrow L^{2^*}(\Rn)$ is continuous (see \cite[Lemma 9]{SerVal2}) .

We set $$||u||^2_{\dot{H}^s(\Rn)}:=\f{c_{N,s}}{2}\int_{\Rn}\int_{\Rn}\f{|u(x)-u(y)|^2}{|x-y|^{N+2s}}dxdy,$$ and we define $\dot{H}^s(\Rn)$ as the completion of $C^{\infty}_0(\Rn)$ w.r.t. the norm $||u||_{\dot{H}^s(\Rn)}+|u|_{L^{2^*}(\Rn)}$ (see \cite{DMV} and \cite{PP}).

\begin{definition}\lab{def-1}
We say that $u\in \dot{H}^s(\Rn)\cap L^{q+1}(\Rn)$ is a weak solution of Eq. \eqref{entire}, if $u>0$ in $\Rn$ and for every $\va\in \dot{H}^s(\Rn)$, 
$$
\int_{\Rn}\int_{\Rn}\f{(u(x)-u(y))(\va(x)-\va(y))}{|x-y|^{N+2s}}dxdy=\int_{\Rn}u^p\va\ dx-\int_{\Rn}u^q\va\ dx$$
or equivalently, 
$$\int_{\Rn}(-\De)^{\f{s}{2}}u(-\De)^{\f{s}{2}}\va\ dx=\int_{\Rn}u^p\va\ dx-\int_{\Rn}u^q\va\ dx.$$
Similarly, when $\Om$ is a bounded domain, we say $u\in X_0\cap L^{q+1}(\Om)$ is a weak solution of Eq. \eqref{eq:a3'} if $u>0$ in $\Om$ and for every $\va\in X_0$, the above integral expression holds.
\end{definition}

\begin{definition}
A positive function $u\in C(\Rn)$ is said to be a classical solution of 
\be\lab{u-f}
(-\De)^s u=f(u) \quad\text{in}\quad \Rn,
\ee
 if $(-\De)^s u$ can be written as \eqref{De-u} and  \eqref{u-f} is satisfied pointwise in all $\Rn$.
\end{definition}

\vspace{3mm}

In recent years, a great deal of attention has been devoted to fractional and non-local operators of elliptic type. One of the main reasons comes from the fact that
this operator naturally arises in several physical phenomenon like flames propagation and chemical reaction of liquids, population
dynamics, geophysical fluid dynamics, mathematical finance etc (see \cite{A, CT, V, VIKH} and the references therein). 
\vspace{3mm}

When $s=1$,  it follows by celebrated Pohozaev identity that \eqref{entire} does not have any solution when $p=2^*-1$ and $q>p$. In this paper we prove this result for all $s\in(0,1)$ by establishing the Pohozaev identity in $\Rn$ for the equation \eqref{entire}. We recall that \eqref{entire} has an equivalent formulation by Caffarelli-Silvestre harmonic extension method in $\mathbb{R}_+^{N+1}$. For spectral fractional laplace equation in bounded domain, some Pohozaev type identities were proved in \cite{BCP, CC, CTan}.  In \cite{FW}, Fall and Weth have proved some nonexistence results associated with the problem $(-\De)^s u=f(x, u)$ in $\Om$ and $u=0$ in $\Rn\setminus\Om$ by applying method of moving spheres. 

Very recently Ros-Oton and Serra \cite[Theorem 1.1]{RS2} have proved Pohozaev identity by direct method for the bounded solution of Dirichlet boundary value problem. More precisely they have proved the following:

Let $u$ be a bounded solution of 
\begin{equation}
  \label{eq:RS-1}
\left\{\begin{aligned}
      (-\De)^s u &=f(u) \quad\text{in }\quad \Om, \\
         u &=0 \quad\text{in}\quad\Rn\setminus\Om,
 \end{aligned}
  \right.
\end{equation}
where $\Om$ is a bounded $C^{1,1}$ domain in $\Rn$, $f$ is locally Lipschitz and $\de(x)=\text{dist}(x,\pa\Om)$. Then $u$ satisfies the following identity:
$$(2s-N)\int_{\Om}uf(u)\ dx+2N\int_{\Om}F(u)\ dx=\Ga(1+s)^2\int_{\pa\Om}(\f{u}{\de^s})^2(x\cdot\nu) dS,$$
where $F(t)=\displaystyle\int_0^t f$ and $\nu$ is the unit outward normal to $\pa\Om$ at $x$ and $\Ga$ is the Gamma function. For nonexistence result with general integro-differential operator we cite \cite{RS3}.

\vspace{3mm}


To apply the technique of \cite{RS2} in the case of $\Om=\Rn$, one needs to know  decay estimate of $u$ and $\na u$ at infinity.  In \cite{RS2}, Ros-Oton and Serra have remarked that assuming certain decay condition of $u$ and $\na u$, one can show that $(-\De)^s u=u^p$ in $\Rn$ does not have any nontrivial solution for $p>\f{N+2s}{N-2s}$. In this article for \eqref{entire}
we first establish decay estimate of $u$ and $\na u$ at infinity and then using that we prove Pohozaev identity for the solution of \eqref{entire} for all $s\in(0,1)$ and consequently we deduce the nonexistence of nontrivial solution when $p=2^{*}-1$. In the appendix, using harmonic extension method in the spirit of Cabr\'{e} and Cinti \cite{CC}, we give an alternative proof of Pohozaev identity in $\Rn$ for the equation of the form 
$$(-\De)^s u=f(u) \quad\text{in}\quad\Rn,$$ where $u\in \dot{H}^s(\Rn)\cap L^{\infty}(\Rn)$ and $f\in C^2$ . The interesting fact about this proof is that, here we do not require decay estimate of $u$ and $\na u$ at infinity as we use suitable cut-off function and in limit we take that cut-off function approaches to $1$.  

On the contrary to the nonexistence result for $p=2^*-1$,  we show using constrained minimization method that Eq.\eqref{entire} admits a positive solution  when $p>2^*-1$. Moreover, we study the qualitative properties of solution. More precisely, using Moser iteration technique  we  prove that any solution, $u$, of \eqref{entire} is in $L^{\infty}(\Rn)$ and we establish decay estimate of $u$ and $\na u$. Then using the Schauder estimate from \cite{RS1} and the $L^{\infty}$ bound that we establish, we show that $u\in C^{\infty}(\Rn)$  if both $p$ and $q$ are integer and $C^{2ks+2s}(\Rn)$, where $k$ is the largest integer satisfying   $\lfloor 2ks\rfloor<p$  if $p\not\in\N$ and $\lfloor 2ks\rfloor<q$ if $p\in\N$ but $q\not\in\N$, where $\lfloor 2ks\rfloor$ denotes the greatest integer less than equal to $2ks$ . We also prove that $u$ is a classical solution.  Thanks to decay estimate of solution that we establish, we further show that solution of \eqref{entire} is radially symmetric. 

When $\Om$ is a bounded domain, we prove that \eqref{eq:a3'} admits a solution for every $p\geq 2^*-1$. For similar type of equations involving critical and supercritical exponents in the case of local operator such as $-\De$, we cite \cite{BS}, \cite{MMPT}-\cite{MP2}. For similar kind of equations with nonlocal operator we cite \cite{BMS, DDW}.

\vspace{3mm}

We turn now to a brief description of the main theorems presented below.

\begin{theorem}\lab{t:est}
Let $s\in(0,1)$, $p\geq 2^*-1$ and $q>(p-1)\f{N}{2s}-1$.  If $u$ is any weak solution of  Eq.\eqref{entire} or Eq.\eqref{eq:a3'},  then  $u\in L^{\infty}(\Rn)$. 
Moreover, if $\Om=\Rn$, then there exist two positive constants $C_1$, $C_2$ such that
\be\lab{sol-est}
C_1|x|^{-(N-2s)}\leq u(x)\leq C_2|x|^{-(N-2s)}, \quad |x|>R_0,
\ee
for some $R_0>0$.
\end{theorem}

 \begin{theorem}\lab{t:reg-1}
Let $s, p, q$ are as in Theorem \ref{t:est}.\\
(i) If $u$ is a weak solution of Eq. \eqref{entire}, then $u\in C^{\infty}(\Rn)$  if both $p$ and $q$ are integer and $u\in C^{2ks+2s}(\Rn)$, where $k$ is the largest integer satisfying   $\lfloor 2ks\rfloor<p$  if $p\not\in\N$ and $\lfloor 2ks\rfloor<q$ if $p\in\N$ but $q\not\in\N$, where $\lfloor 2ks\rfloor$ denotes the greatest integer less than equal to $2ks$ .

\noi(ii) If $u$ is a weak solution of Eq.\eqref{eq:a3'}, then $u\in C^{s}(\Rn)\cap C^{2s+\al}_{loc}(\Om)$, for some $\al\in(0,1)$.  
 \end{theorem}

\begin{theorem}\lab{t:grad-est}   Let $s, p, q$ are as in Theorem \ref{t:est}.   If $u$ is a solution of Eq.\eqref{entire}, then 
\be\lab{grad-est}
|\na u(x)|\leq C|x|^{-(N-2s+1)}, \quad |x|>R',
\ee
for some positive constants $C$ and $R'$.
\end{theorem}

\begin{theorem}\label{t:poho} 
Let $s\in(0,1)$ and $p=2^{*}-1$ and $q>p$. Then \eqref{entire} does not have any solution.
\end{theorem}

\vspace{3mm}

 We define the functional
\be\label{a22} F(v, \Om)= \frac{1}{2}\displaystyle\int_{\R^N}\int_{\Rn}\frac{|u(x)-u(y)|^2}{|x-y|^{N+2s}} dxdy+\frac{1}{q+1}\displaystyle\int_{\Om}|v|^{q+1} dx. \ee

Define,
\be\lab{K}
\mathcal{K}:= \inf\bigg\{F(v, \Rn): v\in \dot{H}^s(\Rn)\cap L^{q+1}(\Rn), \int_{\Rn} |v|^{p+1}dx=1\bigg\}.
\ee

\begin{theorem}\label{1.2}
Let $s\in(0,1)$ and $q>p> 2^*-1$.  Then $\mathcal{K}$ in \eqref{K} is achieved by a radially decreasing function  $u\in \dot{H}^s(\R^N) \cap L^{q+1}(\R^N)$ and Eq.\eqref{entire} admits a nonnegative solution. Furthermore, if $q>(p-1)\f{N}{2s}-1$, then Eq. \eqref{entire} admits a positive solution.

\end{theorem}

\vspace{3mm}

When $\Om$ is a smooth bounded domain, we define
\be\lab{S-Om} S_{\Om}:= \inf\bigg\{F(v, \Om): v\in X_0\cap L^{q+1}(\Om), \int_{\Om} |v|^{p+1}dx=1\bigg\}.
\ee

\begin{theorem}\label{1.3}
 Let $s\in(0,1)$ and $q>p\geq 2^*-1$.  Then $\mathcal{S}_{\Om}$ in \eqref{S-Om} is achieved by a function  $u\in X_0 \cap L^{q+1}(\Om)$.   Furthermore, there exists a constant $\lambda>0$, such that $u$ satisfies
\begin{equation}
  \label{a4}
\left\{\begin{aligned}
      (-\Delta)^s u &=\la |u|^{p-1}u -|u|^{q-1}u \quad\text{in }\quad \Om,\\
       u&=0 \quad\text{in}\quad\Rn\setminus\Om.
         \end{aligned}
  \right.
\end{equation}
Furthermore, if $p\geq 2^*-1$ and $q>(p-1)\f{N}{2s}-1$, then Eq.\eqref{a4} admits a positive solution.

Note that the scaled function $U=\la^\f{1}{p-1}u$ satisfies the equation 
\be\lab{eq:U}
(-\De)^s U=U^{p}-c^*U^{q}, \quad c^*=\la^{-\f{q-1}{p-1}}. 
\ee
\end{theorem}

We organise the paper as follows. In section 2, we recall equivalent formulation of \eqref{entire}  by the Caffarelli-Silvestre \cite{CS} associated extension problem-a local PDE in $\mathbb{R}^{N+1}_+$ and we also recall Schauder estimate for the nonlocal equation proved by Ros-Oton and Serra \cite{RS1}. In Section 3, we establish $u\in L^{\infty}(\Rn)$, decay estimate of solution and the gradient of solution at infinity. Section 4 deals with the proof of nonexistence result in $\Rn$ when $p=2^*-1$. In section 5, we show that any solution of \eqref{entire} is radially symmetric and strictly decreasing about some point in $\Rn$. While in section 6, we prove existence of solution to \eqref{entire} for $p>2^*-1$ and to \eqref{eq:a3'} when $\Om$ is bounded and $p\geq 2^*-1$.

\vspace{3mm}

{\bf Notations:} Throughout this paper we use the notation $C^{\ba}(\Rn)$, with $\ba>0$ to refer the space $C^{k, \ba'}(\Rn)$, where $k$ is the greatest integer such that 
$k<\ba$ and $\ba'=\ba-k$. According to this, $[.]_{C^{\ba}(\Rn)}$ denotes the following seminorm 
$$[u]_{C^{\ba}(\Rn)}=[u]_{C^{k, \ba'}(\Rn)}=\sup_{x,\ y\in\Rn, x\not=y}\f{|D^ku(x)-D^ku(y)|}{|x-y|^{\ba'}}.$$
Throughout this paper, $C$  denotes the generic constant, which may vary from line to line and ${\bf n}$ denotes the unit outward normal.

\section{Preliminaries}
In this section we recall the other useful representation of fractional laplacian $(-\De)^s$, which we will use to prove decay estimate of solution at infinity. Using the celebrated  Caffarelli and Silvestre extension method, (see \cite{CS}), fractional laplacian $(-\De)^s$ can be seen as a trace class operator (see \cite{CS, GS, BCPS}) .
Let $u \in \dot{H}^s(\Rn)$ be a solution of \eqref{entire}.  Define $w:=E_{s}(u)$ be its $s-$ harmonic extension to the upper half space $\R^{N+1}_+, $
that is, there is a solution to the following problem:
\begin{align}\lab{A-41}
 \begin{cases}
  \mbox{div}(y^{1-2s} \na w)=0  &\quad\mbox{in}\quad \R^{N+1}_+,\\
  w=u &\quad\mbox{on}\quad\mathbb{R}^N\times \{y=0\}.
 \end{cases}
\end{align}

Define the space  
$X^{2s} (\R_+^{N+1}):=$ closure of $C_0^\infty(\overline{\R_+^{N+1}})$ w.r.t. the following norm 
$$\norm{w}_{2s}=\norm{w}_{X^{2s} (\R_+^{N+1})}:=\bigg(k_{2s}\int_{\R_+^{N+1}}y^{1-2s}|\bigtriangledown w|^2dxdy\bigg)^\frac{1}{2},$$
where $k_{2s}=\frac{\Ga(s)}{2^{1-2s}\Ga(1-s)}$ is a normalizing constant, chosen in such a way that the 
extension operator 
$E_{s}:\dot{H}^{s}(\Rn) \To X^{2s} (\R_+^{N+1})$ is an isometry (up to constants), that is, 
$\norm{E_{s} u}_{2s}=\norm{u}_{\dot{H}^{s}(\Rn)}=|\left(-\Delta\right)^{s} u|_{L^2(\Rn)}. $ (see \cite{DMV}).
Conversely, for a function $w \in X^{2s} (\R_+^{N+1})$, we denote its trace on $\Rn \times \{y=0\}$ as:
$$\mbox{Tr}(w):=w(x,0). $$
This trace operator satisfies:
\begin{align}\lab{tr-ineq}
\norm{w(.,0)}_{\dot{H}^{s}(\Rn)}=\norm{Tr(w)}_{\dot{H}^{s}(\Rn)} \leq \norm{w}_{2s}.
\end{align}
Consequently,
\be\lab{tr-ineq1}\bigg(\int_{\Rn}|u(x)|^{2^*}dx\bigg)^{\frac{2}{2^*}} \leq S(N,s)\int_{\R^{N+1}_+}y^{1-2s}|\na w(x,y)|^2dxdy.
\ee
Inequality \eqref{tr-ineq1} is called the trace inequality.
We  note that $H^1(\R^{N+1}_+,y^{1-2s})$, up to a normalizing factor, is isometric to $X^{2s} (\R_+^{N+1})$ (see \cite{GS}). 
In \cite{CS}, it is shown that $E_{s}(u)$ satisfies the following:
$$(-\De)^{s}u(x)=\frac{\pa w}{\pa \nu ^{2s} }:=-k_{2s} \lim_{y \to 0^+}y^{1-2s}\frac{\pa w}{\pa y}(x,y). $$
With this above representation,  $\eqref{A-41}$ can be rewritten as:
\begin{equation}\lab{A-42}
\left\{\begin{aligned}
      \text{div}(y^{1-2s} \na w) &=0  \quad\text{in}\quad \R^{N+1},\\
  \frac{\pa w}{\pa \nu ^{2s} } &=w^p(.,0)-w^{q}(.,0) \quad\text{on}\quad\mathbb{R}^N.
          \end{aligned}
  \right.
\end{equation}

A function $w \in X^{2s} (\R_+^{N+1})$ is said to be a weak solution to \eqref{A-42} if for all $\va \in X^{2s} (\R_+^{N+1}), $
we have
\begin{equation}\lab{A-43}
k_{2s} \int_{\R^{N+1}_+}y^{1-2s}\na w \na \va\ dxdy  =\int_{\Rn}w^p(x,0) \va (x,0)\ dx- \int_{\Rn}w^{q}(x,0) \va(x,0)\ dx.
\end{equation}
Note that for any weak solution $w \in X^{2s} (\R_+^{N+1})$ to \eqref{A-42}, the function 
$u:=\mbox{Tr}(w)=w(.,0)\in \dot{H}^{s}(\Rn)$ is a weak solution to \eqref{entire}. 

Next, we recall Schauder estimate for the nonlocal equation by Ros-Oton and Serra \cite{RS1}.
\begin{theorem}{[Ros-Oton and Serra, 2016]}\lab{t:RS}
Let $s\in(0,1)$ and $u$ be any bounded weak solution to
$$(-\De)^s u=f \quad\text{in}\quad B_1(0).$$
Then,\\
(a) If $u\in L^{\infty}(\Rn)$ and $f\in L^{\infty}(B_1(0))$,
$$||u||_{C^{2s}(B_{\f{1}{2}}(0))}\leq C(||u||_{L^{\infty}(\Rn)}+||f||_{L^{\infty}(B_1(0))}) \quad\text{if}\quad s\not=\f{1}{2}$$ 
and
$$||u||_{C^{2s-\eps}(B_{\f{1}{2}}(0))}\leq C(||u||_{L^{\infty}(\Rn)}+||f||_{L^{\infty}(B_1(0))})\quad\text{if}\quad s=\f{1}{2},$$
for all $\eps>0$.

(b) If $f\in C^{\al}(B_1(0))$ and $u\in C^{\al}(\Rn)$ for some $\al>0$, then
$$||u||_{C^{\al+2s}(B_{\f{1}{2}}(0))}\leq C(||u||_{C^{\al}(\Rn)}+||f||_{C^{\al}(B_1(0))}) ,$$
whenever $\al+2s$ is not an integer.
The constant $C$ depends only on $N, s, \al, \eps$.
\end{theorem} 
We conclude this section by recalling some weighted embedding results from Tan and Xiong \cite{TX}. For this, we introduce the following notations
$$Q_R=B_R\times [0, R)\subset\R^{N+1},$$
where $B_R$ is a ball in $\Rn$ with radius $R$ and centered at origin. Note that, $B_R\times\{0\}\subset Q_R$. We define,
$$H(Q_R, y^{1-2s}):=\displaystyle\bigg\{U\in H^1(Q_R): \ \int_{Q_R}y^{1-2s}(U^2+|\na U|^2)dxdy<\infty \bigg\}$$ and $X^{2s}_0(Q_R)$ is the closure of $C^{\infty}_0(Q_R)$ with respect to the norm $$||w||_{X^{2s}_0(Q_R)}=\bigg(\int_{Q_R}y^{1-2s}|\na w|^2dxdy\bigg)^\f{1}{2}.$$

We note that, $s \in (0,1)$ implies the weight $y^{1-2s}$ belongs to the Muckenhoupt class $A_2$ (see \cite{Muckenhoupt})
which consists of all non-negative functions $w$ on $\R^{N+1}$ satisfying for some constant $C$, the estimate
\begin{eqnarray*}
\sup_B\bigg(\frac{1}{|B|}\int_B wdx\bigg)\bigg(\frac{1}{|B|}\int_B w^{-1}dx\bigg) \leq C, 
\end{eqnarray*}
where the supremum is taken over all balls $B$ in $\R^{N+1}. $

\begin{lemma}\lab{l:embed-1}
Let $f\in X^{2s}_0(Q_R)$. Then there exists constant $C$ and $\delta >0$ depending only on $N$ and $s$ such that for any $1 \leq k \leq \frac{n+1}{n}+\delta, $
$$\bigg(\int_{Q_R}y^{1-2s}|f|^{2k}dxdy\bigg)^{\frac{1}{2k}} \leq C(R) \bigg(\int_{Q_R}y^{1-2s}|\na f|^2dxdy\bigg)^{\frac{1}{2}}.$$
\end{lemma}
\begin{proof}
It is known from \cite[Lemma 2.1]{TX} that the lemma holds for $f\in C^1_c(Q_R)$ (also see \cite{FKS}). For general $f$, the lemma can be easily proved applying density argument and Fatou's lemma.
\end{proof}

  \begin{lemma}\label{l:embed-2}
   Let $f \in X^{2s}_0(Q_R)$. Then there exists a positive constant $\delta$ depending only on $N$ and $s$ such that
   $$\int_{B_R \times \{y=0\}}|f|^2dx \leq \eps \int_{Q_R}y^{1-2s}|\na f|^2dxdy +\frac{C(R)}{\eps^{\delta}}\int_{Q_R}y^{1-2s}|f|^2dxdy,$$ for any $\eps>0$.
  \end{lemma}
\begin{proof}
If $f\in C^1_c(Q_R)$, then the lemma holds (see \cite[Lemma 2.3]{TX}). For $f \in X^{2s}_0(Q_R)$,  there exists $f_n\in C^{\infty}_0(Q_R)$ such that $f_n\to f$ in $||.||_{X^{2s}_0(Q_R)}$ and for $f_n$, we have
\be\lab{A-81}\int_{B_R \times \{y=0\}}|f_n|^2dx \leq \eps \int_{Q_R}y^{1-2s}|\na f_n|^2dxdy +\frac{C(R)}{\eps^{\delta}}\int_{Q_R}y^{1-2s}|f_n|^2dxdy,
\ee
 for any $\eps>0$. Clearly the 1st integral on RHS converges to $\displaystyle\int_{Q_R}y^{1-2s}|\na f|^2dxdy$. Thanks to Lemma \ref{l:embed-1}, it follows that the embedding $X^{2s}_0(Q_R)\hookrightarrow L^2(Q_R, y^{1-2s})$ is continuous. Therefore,  we can also pass to the limit in the 2nd integral of the RHS. On the other hand, using the trace embedding result, we can also pass to the limit on LHS. Hence, the lemma follows.
 \end{proof}

\section{$L^{\infty}$ estimate and decay estimates}
{\bf Proof of Theorem \ref{t:est}}\begin{proof} 
{\bf Case 1:} Suppose $\Om=\Rn$.

\vspace{2mm}

Let $u$ be an arbitrary weak solution of Eq.\eqref{entire}. We first prove that $u\in L^{\infty}_{loc}(\Rn)$ by Moser iterative technique (see, for example \cite{JLX, TX}).
 From Section-2, we know that $w(x,y)$, the $s-$harmonic extension of $u$, is a solution of \eqref{A-42}. 

Let $B_{r}$ denote the ball in $\R^{N}$ of radius $r$ and centered at origin. We define
$$Q_{r}=B_{r} \times [0,r).$$
Set $\bar w=w^++1$ and for $L>1$, define
\begin{equation*}
w_L=
\left\{\begin{aligned}
      &\bar w \quad&&\text{if}\quad w<L\\
       &1+L  \quad&&\text{if }\quad w\geq L.
         \end{aligned}
  \right.
\end{equation*}
For $t>1 $, we choose the test function $\va$ in \eqref{A-43} as follows:
\be\lab{A-91}\va(x,y)=\eta^2(x, y)\big(\bar w(x,y) w_L^{2(t-1)}(x,y)-1\big),\ee where
$ \eta\in C^{\infty}_0(Q_R)$ with  $0\leq\eta\leq 1$,  
 $\eta=1$ in $Q_r$, $0<r<R\leq 1$ and $|\na \eta| \leq \frac{2}{R-r}$.
Note that $\va \in X^{2s}(\R^{N+1}_{+})$. Using this test function $\va$, we obtain from \eqref{A-43}

\begin{align}\label{w frm}
&k_{2s}\int_{\R^{N+1}_+}y^{1-2s}\na w(x,y) \na {\bigg(\eta^2(x,y)\big(\bar w(x,y)w_L^{2(t-1)}(x,y)-1\big)\bigg)}dxdy\no\\
&\qquad=\int_{\Rn} \big(w^p(x,0)-w^q(x,0)\big)\eta^2(x,0)\big(\bar w(x,0)w_L^{2(t-1)}(x,0)-1\big)dx.
\end{align}
Direct calculation yields
\be\lab{A-83}\na \big(\eta^2(\bar w w_L^{2(t-1)}-1)\big)=2\eta (\bar w w_L^{2(t-1)}-1) \na \eta + \eta^2 w_L^{2(t-1)} \na\bar w+2(t-1)\eta^2 \bar ww_L^{2(t-1)-1} \na w_L. \ee
Here we observe that on the set $\{w<0\}$, we have $\va=0$ and $\na\va=0$. Thus $(\ref{w frm})$ remains  same if we change the domain of integration to $\{w\geq 0\}$. Therefore, in the support of the integrand $\na w=\na\bar w$. As a result,
substituting \eqref{A-83} into $(\ref{w frm})$, it follows
\begin{align*}
&k_{2s}\int_{\R^{N+1}_+}y^{1-2s}\bigg(2 \eta(\bar ww_L^{2(t-1)}-1)\na \eta \na\bar w\no\\
&\qquad+ \eta^2 w_L^{2(t-1))}\na\bar w\na w+2(t-1)\eta^2w_L^{2(t-1)-1}\bar w \na w_L\na w\bigg)(x,y)dxdy\no\\
&\qquad\qquad\qquad\qquad\leq \int_{\Rn} \eta^2(x,0)w^{p}(x,0)\bar w(x,0)w_L^{2(t-1)}(x,0)dx.\no
\end{align*}
Notice that in the support of the integrand of second integral on the LHS $\na\bar w=\na w$   and in the third integral $w_L=\bar w$, $\na w_L=\na w$. Hence the above expression reduces to
\begin{align}\label{grd wfrm}
&k_{2s}\int_{\R^{N+1}_+}y^{1-2s}\bigg(2 \eta(\bar ww_L^{2(t-1)}-1)\na \eta \na\bar w\no\\
&\qquad+ \eta^2 w_L^{2(t-1))}|\na\bar w|^2+2(t-1)\eta^2w_L^{2(t-1)} |\na w_L|^2\bigg)(x,y)dxdy\no\\
&\qquad\qquad\qquad\qquad\leq \int_{\Rn} \eta^2(x,0)\bar w^{p+1}(x,0)w_L^{2(t-1)}(x,0)dx,
\end{align}
where for the RHS, we have used the fact that $w\leq \bar w$.

Using Young's inequality we have,
\begin{align}\label{*2}
\big|2\eta(\bar ww_L^{2(t-1)}-1) \na \eta \na\bar w \big|
\leq \frac{1}{2}\eta^2w_L^{2(t-1)}|\na\bar w|^2+2\bar w^2w_L^{2(t-1)}|\na \eta|^2.
\end{align}
Using \eqref{*2}, from \eqref{grd wfrm} we obtain,
\begin{align} \label{*3}
\f{k_{2s}}{2}\int_{\R^{N+1}_+}y^{1-2s}\bigg(|\na\bar w|^2+(t-1)|\na w_L|^2\bigg)\eta^2 w_L^{2(t-1)}(x,y)dxdy\no\\
\qquad\qquad \leq 2k_{2s}\int_{\R^{N+1}_+}y^{1-2s}\bar w^2w_L^{2(t-1)}|\na \eta|^2(x,y)dxdy\no\\
+\int_{\Rn}\bar w^{p+1}w_L^{2(t-1)}\eta^2(x,0)dx.
\end{align}
As $t>1$ and $\na w_L=0$ for $w \geq L$ , it is not difficult to observe that,
\begin{align}\lab{A-51}
&\int_{\R^{N+1}_+}y^{1-2s}|\na (\eta\bar w w_L^{t-1})|^2 dxdy\no\\
&\leq 3\int_{\R^{N+1}_+}y^{1-2s}\left(\bar w^2w_L^{2(t-1)}|\na \eta|^2+\eta^2 w_L^{2(t-1)}|\na\bar w|^2+(t-1)^2\eta^2w_L^{2(t-1)}|\na w_L|^2\right)dxdy\no\\
&\leq3t\int_{\R^{N+1}_+}y^{1-2s}\bar w^2w_L^{2(t-1)}|\na \eta|^2dxdy\no\\
&+3t\int_{\R^{N+1}_+}y^{1-2s}\bigg(|\na\bar w|^2+(t-1)|\na w_L|^2\bigg)\eta^2 w_L^{2(t-1)}dxdy.
\end{align}

Combining \eqref{A-51} and \eqref{*3}, we have
\begin{eqnarray}\lab{A-55}
&&k_{2s}\displaystyle\int_{\R^{N+1}_+}y^{1-2s}|\na (\eta\bar w w_L^{t-1})|^2 dxdy \no\\
&&\leq 3tk_{2s}\displaystyle\int_{\R^{N+1}_+}y^{1-2s}\bar w^2w_L^{2(t-1)}|\na \eta|^2 dxdy\no\\
&&+3t\Big\{4k_{2s}\displaystyle\int_{\R^{N+1}_+}y^{1-2s}\bar w^2w_L^{2(t-1)}|\na \eta|^2(x,y)dxdy
+2\int_{\Rn}\bar w^{p+1}w_L^{2(t-1)}\eta^2(x,0)dx\Big\}.\no\\
\end{eqnarray}
For $p\geq 2^*-1$, choose $\al>1$ as follows:
\begin{equation}\label{eq:p-t}
\frac{N}{2s}<\al<\frac{q+1}{p-1}.
\end{equation}
Note that for $p=2^*-1$ the interval $(\f{N}{2s}, \f{q+1}{p-1})$ is always a nonempty set. On the other hand, as $q>(p-1)\f{N}{2}-1$, it follows $(\f{N}{2s}, \f{q+1}{p-1})\not=\emptyset$, when $p>2^*-1$. From \eqref{eq:p-t} we have,
$$(p-1)\al<q+1 \quad\text{and}\quad 2<\frac{2\al}{\al-1}<2^*.$$ 
As supp$(\eta(\cdot, 0))\subset B_R$ and $w(x,0)= u\in L^{q+1}(\Rn)$, it follows $\bar w(.,0)=w^{+}(x,0)+1=u+1\in L^{q+1}(B_1)$.  This along with the fact that supp $\eta\subset Q_R$, where $R<1$, we obtain

\begin{align}\label{pth trm}
 \int_{\Rn}\bar w^{p+1}w_L^{2(t-1)}\eta^2(x,0)dx &= \int_{B_1}\bar w^{p+1}w_L^{2(t-1)}\eta^2(x,0)dx\no\\
 &=\int_{B_1}|\eta\bar w w_L^{(t-1)}(x,0)|^2\bar w^{p-1}(x,0)dx\no\\
 &\leq\left(\int_{B_1}\bar w^{\al(p-1)}(x,0)dx\right)^{\f{1}{\al}}
 \left(\int_{B_R}|\eta\bar w w_L^{(t-1)}|^{\f{2\al}{\al-1}}(x,0)dx\right)^{\f{\al-1}{\al}}\no\\
 &\leq C \|\eta\bar w w_L^{(t-1)}\|_{L^{\f{2\al}{\al-1}}(B_R)}^2.
\end{align}
By interpolation inequality,
\begin{align}\lab{A-54}
 \|\eta\bar w w_L^{(t-1)}\|_{L^{\f{2\al}{\al-1}}(B_R)}^2\leq
 \|\eta\bar w w_L^{(t-1)}\|_{L^{2}(B_R)}^{2\theta} \|\eta\bar w w_L^{(t-1)}\|_{L^{2^*}(B_R)}^{2(1-\theta)},
\end{align}
where $\theta$ is determined by \be\lab{5-theta}\f{\al-1}{2\al}=\f{\theta}{2}+\f{1-\theta}{2^*}.\ee
Applying Young's inequality,  \eqref{A-54} yields
\begin{align}\lab{A-61}
 \|\eta\bar w w_L^{(t-1)}\|_{L^{\f{2\al}{\al-1}}(B_R)}^2&\leq
 C(s,\al,N)\eps^2 \|\eta\bar w w_L^{(t-1)}\|_{L^{2^*}(\mathbb{R}^N)}^{2} \no\\
  &+ C(\al,s,N) \eps^{-\f{2(1-\theta)}{\theta}}\|\eta\bar w w_L^{(t-1)}\|_{L^{2}(B_R)}^2.
\end{align}
Therefore, using Sobolev Trace inequality \eqref{tr-ineq} and the value of $\theta$ from \eqref{5-theta}, we have
\begin{align}\label{trc}
 \|\eta\bar w w_L^{(t-1)}\|_{L^{\f{2\al}{\al-1}}(B_R)}^2&\leq
 C(s,\al,N)\eps^2 \int_{\mathbb{R}_+^{N+1}}y^{1-2s}|\na\left(\eta\bar w w_L^{(t-1)}\right)|^2dxdy \no\\
 &\qquad+ C(\al,s,N) \eps^{-\f{2N}{2\al s-N}}\int_{B_R}|\eta\bar w w_L^{(t-1)}(x,0)|^2dx.
\end{align}
Thanks to Lemma \ref{l:embed-2}, for $\delta>0$ we have
\bea\label{ppr lm}
 \int_{B_R}|\eta\bar w w_L^{(t-1)}(x,0)|^2dx &=&\int_{B_1}|\eta w w_L^{(t-1)}(x,0)|^2dx\no\\
 & \leq&\delta \int_{Q_1}y^{1-2s}|\na\left(\eta\bar w w_L^{(t-1)}\right)|^2dxdy \no\\
 &+& \f{C}{\delta^\ba}\int_{Q_1}y^{1-2s}|\eta\bar w w_L^{(t-1)}|^2dxdy,
\eea
where  $\ba=\f{s'+1}{s'-1}$, with some $1<s'<\f{1}{1-s}.$ Substituting  $(\ref{ppr lm})$ in $(\ref{trc})$ and then  $(\ref{trc})$ in  $(\ref{pth trm})$ yields
\begin{align}\label{pth trm1}
 \int_{\Rn}\bar w^{p+1}w_L^{2(t-1)}\eta^2(x,0)dx 
 &\leq C(s,\al,N)\eps^2 \int_{\mathbb{R}_+^{N+1}}y^{1-2s}|\na\left(\eta\bar w w_L^{(t-1)}\right)|^2dxdy \no\\
 &+ C(\al,s,N) \eps^{-\f{2N}{2\al s-N}}\delta \int_{\R^{N+1}_+}y^{1-2s}|\na\left(\eta\bar w w_L^{(t-1)}\right)|^2dxdy \no\\
 &+ C(\al,s,N) \eps^{-\f{2N}{2rs-N}}\f{1}{\delta^\ba}\int_{\R^{N+1}_+}y^{1-2s}|\eta\bar w w_L^{(t-1)}|^2dxdy.
\end{align}
Consequently, substituting \eqref{pth trm1} in \eqref{A-55}, we obtain
\begin{align}\label{intg2}
\int_{\R^{N+1}_+}y^{1-2s}|\na (\eta\bar w w_L^{t-1})|^2 dxdy
&\leq Ct\int_{\R^{N+1}}y^{1-2s}\bar w^2w_L^{2(t-1)}|\na \eta|^2 dxdy \no\\
&+Ct\bigg(\eps^2+\eps^{-\f{2N}{2\al s-N}}\de\bigg) \int_{\mathbb{R}_+^{N+1}}y^{1-2s}|\na\left(\eta\bar w w_L^{(t-1)}\right)|^2dxdy \no\\
 &+ Ct\eps^{-\f{2N}{2\al s-N}}\de^{-\ba}\int_{\R^{N+1}}y^{1-2s}|\eta\bar w w_L^{(t-1)}|^2dxdy. 
\end{align}
Choose $$\eps= \f{1}{2\sqrt{Ct}} \quad\text{and}\quad  \delta=\f{\eps^\f{2N}{2\al s-N}}{4Ct}.$$ 
Hence,  from $(\ref{intg2})$, a direct calculation yields  
\begin{eqnarray}\label{intg3}
\f{1}{2}\int_{\R^{N+1}_+}y^{1-2s}|\na (\eta\bar w w_L^{t-1})|^2 dxdy
 &\leq& Ct\int_{\R^{N+1}_+}y^{1-2s}\bar w^2w_L^{2(t-1)}|\na \eta|^2 dxdy \no\\
 &+& Ct^\f{2\al s(\ba+1)}{2\al s-N} \int_{\R^{N+1}_+}y^{1-2s}|\eta\bar w w_L^{(t-1)}|^2dxdy\no\\
&\leq& Ct^{\ga} \int_{\R^{N+1}_+}y^{1-2s}\left(\eta^2+|\na\eta|^2\right)\bar w^2 w_L^{2(t-1)}dxdy.
\end{eqnarray}
where $\ga=\f{2\al s(\ba+1)}{2\al s-N}.$ 
Applying Sobolev inequality  (see Lemma \ref{l:embed-1}), we obtain from \eqref{intg3}
\begin{eqnarray*}
\left(\int_{Q_1}y^{1-2s}|\eta\bar w w_L^{t-1}|^{2\chi} dxdy\right)^{\f{1}{\chi}}
&\leq& C\int_{Q_1}y^{1-2s}|\na (\eta\bar w w_L^{t-1})|^2 dxdy\no\\
&\leq& Ct^{\ga} \int_{Q_1}y^{1-2s}\left(\eta^2+|\na\eta|^2\right)\bar w^2 w_L^{2(t-1)}dxdy,
\end{eqnarray*}
where $\chi=\f{N+1}{N}>1$. Now using the fact that $0<r<R<1$, $\eta=1$ in $Q_r$, $|\na\eta|\leq\f{2}{R-r}$ and supp $\eta=Q_R$, we get
\begin{align*}
\bigg(\int_{Q_r}y^{1-2s}\bar w^{2\chi} w_L^{2(t-1)\chi} dxdy\bigg)^\f{1}{\chi}
 &\leq \f{Ct^\ga}{(R-r)^2} \int_{Q_R}y^{1-2s}\bar w^2 w_L^{2(t-1)}dxdy.
\end{align*}
As $w_L\leq \bar w$, the above expression yields,
\begin{align*}
\left(\int_{Q_r}y^{1-2s} w_L^{2t\chi } dxdy\right)^{\f{1}{\chi}}
 &\leq \f{Ct^\ga}{(R-r)^2} \int_{Q_R}y^{1-2s}\bar w^{2t}dxdy,
\end{align*}
provided the right-hand side is bounded. Passing to the limit $L\to\infty$ via Fatou's lemma we obtain
\be
\left(\int_{Q_r}y^{1-2s}\bar w^{2t\chi } dxdy\right)^{\f{1}{\chi}}\leq \f{C t^\ga}{(R-r)^2} \int_{Q_R}y^{1-2s}\bar w^{2t}dxdy,\no
 \ee
 that is,
 \be\label{intg4}
 \left(\int_{Q_r}y^{1-2s}\bar w^{2t\chi } dxdy\right)^{\f{1}{2\chi t}}
 \leq \left(\f{Ct^\ga}{(R-r)^2}\right)^{\f{1}{2t}} \left(\int_{Q_R}y^{1-2s}\bar w^{2t}dxdy\right)^{\f{1}{2t}}.
 \ee
Now we  iterate the above relation.  We take $t_i=\chi^i$ and $r_i=\f{1}{2}+\f{1}{2^{i+1}}$ for $i=0,1,2,\dots$
Note that $t_i=\chi t_{i-1}$, $r_{i-1}-r_i=\f{1}{2^{i+1}}$. Hence from \eqref{intg4}, with
$t=t_i$, $r=r_i$, 
$R=r_{i-1}$, we have
\begin{align*}
 \left(\int_{Q_{r_i}}y^{1-2s}\bar w^{2t_{i+1}} dxdy\right)^{\f{1}{2 t_{i+1}}}
 &\leq C^{\f{i}{\chi^i}}\left(\int_{Q_{r_{i-1}}}y^{1-2s} \bar w^{2t_{i}} dxdy\right)^{\f{1}{2 t_{i}}}, \quad i=0, 1, 2, \cdots,
\end{align*}
where $C$ depend only on $N, s, p, q$.   Hence, by iteration we have
\begin{align*}
 \left(\int_{Q_{r_i}}y^{1-2s}\bar w^{2t_{i+1}} dxdy\right)^{\f{1}{2 t_{i+1}}}
 &\leq C^{\sum\f{i}{\chi^i}}\left(\int_{Q_{r_0}}y^{1-2s}\bar w^{2t_0} dxdy\right)^{\f{1}{2 t_0}}, \quad i=0, 1, 2, \cdots,
\end{align*}
Letting $i\to\infty$ we have
\begin{align*}
 \sup _{Q_{\f{1}{2}}}\bar w \leq C|\bar w|_{L^{2}(Q_1,y^{1-2s})},
\end{align*}
which in turn implies
\begin{align*}
  \sup _{B_{\f{1}{2}}}u=\sup _{B_{\f{1}{2}}}w^{+}\leq  \sup _{Q_{\f{1}{2}}}w^+  \leq C\|w\|_{L^{2}(Q_1,y^{1-2s})}.
\end{align*}
Hence, $u\in L^{\infty}(B_{\f{1}{2}}(0)).$ Translating the equation, similarly it follows that $u\in L^{\infty}_{loc}(\Rn)$.

\vspace{3mm}

To show the $L^\infty$ bound at infinity, we define the Kelvin transform of $u$ by the function $\tilde u$ as follows:
$$\tilde u(x)=\f{1}{|x|^{N-2s}}u(\f{x}{|x|^2}), \quad x\in\Rn\setminus\{0\}.$$
It follows from \cite[Proposition A.1]{RS4}, 
\be\lab{Kel}
(-\De)^s\tilde u(x)=\f{1}{|x|^{N+2s}}(-\De)^s u\bigg(\f{x}{|x|^2}\bigg).
\ee
Thus \bea 
(-\De)^s\tilde u(x)&=&\f{1}{|x|^{N+2s}}\displaystyle\left(u^p(\f{x}{|x|^2})-u^q(\f{x}{|x|^2})\right)\no\\
&=& \f{1}{|x|^{N+2s}}\displaystyle\left(|x|^{p(N-2s)}\tilde u^p(x)-|x|^{q(N-2s)}\tilde u^q(x)\right).\no
\eea
This implies $\tilde u$ satisfies the following equation 
\begin{equation}
  \label{eq:til-u}
\left\{\begin{aligned}
      (-\De)^s \tilde u &=|x|^{p(N-2s)-(N+2s)}\tilde u^p -|x|^{q(N-2s)-(N+2s)}\tilde u^q \quad\text{in }\quad \Rn, \\
      \tilde u &\in \dot{H}^s(\Rn)\cap L^{q+1}(\Rn, |x|^{(N-2s)(q+1)-2N}), \\
      \tilde u&>0 \quad\Rn.
 \end{aligned}
  \right.
\end{equation}
That is, \be\lab{A-15-1}(-\De)^s \tilde u= f(x, \tilde u)  \quad\text{in}\quad \Rn,\ee
where \be\lab{A-15-2}f(x, \tilde u) :=|x|^{p(N-2s)-(N+2s)}\tilde u^p -|x|^{q(N-2s)-(N+2s)}\tilde u^q.\ee
Since  $q>p\geq \f{N+2s}{N-2s}$, we get $(-\De)^s \tilde u\leq \tilde u^p$ in $(B_1(0))$. Applying the Moser iteration technique along the same line of arguments as above with a suitable modification, we get $\sup_{B_\rho(0)}\tilde u\leq C$, for some $\rho>0$ and $C$ is a positive constant. This in turn implies, 
\be\lab{up-est-1} u(x)\leq \f{C}{|x|^{N-2s}}\quad |x|>R_0 ,\ee for some large $R_0$. Hence, $u\in L^\infty(\Rn)$. As a consequence $\tilde u\in L^\infty(\Rn)$ and therefore $(-\De)^s\tilde u\in L^\infty(B_1(0))$. Applying Theorem \ref{t:RS}, it follows that $\tilde u\in C(B_{\f{1}{2}}(0))$. Thus there exists $C_1>0$ such that $\tilde u>C_1$ in $(B_{\f{1}{2}}(0))$, which in turn implies $u(x)>\f{C_1}{|x|^{N-2s}}$, for $|x|>2$. This along with \eqref{up-est-1}, yields  \eqref{sol-est} .

\vspace{3mm}

{\bf Case 2:} $\Om$ is a bounded domain. 

\vspace{2mm}

Arguing along the same line with minor modifications, it can be shown that  $u\in L^{\infty}(\Om)$. Therefore the conclusion follows as $u=0$ in $\Rn\setminus\Om$.
 \end{proof}
 
 \vspace{3mm}
  
{\bf Proof of Theorem \ref{t:reg-1}:} \begin{proof}
(i) From Theorem \ref{t:est}, we know any solution $u$ of Eq.\eqref{entire} is in $L^\infty(\Rn)$. Therefore, we have 
\be\lab{1}
(-\De)^s u=f(u), \quad f(u):=u^p-u^q \in L^\infty(\Rn).
\ee As a result, applying Theorem \ref{t:RS}(a) , we obtain
\bea\lab{16}
||u||_{C^{2s}(B_{\f{1}{2}}(0))}&\leq& C(||u||_{L^{\infty}(\Rn)}+||f(u)||_{L^{\infty}(B_1(0))}) \no\\
&\leq& C(||u||_{L^{\infty}(\Rn)}+||f(u)||_{L^{\infty}(\Rn)}) \quad\text{if}\quad s\not=\f{1}{2},
\eea
\bea\lab{17}
||u||_{C^{2s-\eps}(B_{\f{1}{2}}(0))}&\leq& C(||u||_{L^{\infty}(\Rn)}+||f(u)||_{L^{\infty}(B_1(0))})\no\\
&\leq& C (||u||_{L^{\infty}(\Rn)}+||f(u)||_{L^{\infty}(\Rn)}) \quad\text{if}\quad s=\f{1}{2},
\eea
for all $\eps>0$. Here the constants $C$ are independent of $u$, but may depend on radius $\f{1}{2}$ and centre $0$.  Since  the equation is invariant under translation, translating the equation, we obtain 
\bea\lab{18}
||u||_{C^{2s}(B_{\f{1}{2}}(y))} &\leq& C(||u||_{L^{\infty}(\Rn)}+||f(u)||_{L^{\infty}(\Rn})  \no\\
&\leq& C(1+||u||_{L^{\infty}(\Rn)})^q \quad\text{when}\quad s\not=\f{1}{2},
\eea
\be\lab{19} ||u||_{C^{2s-\eps}(B_{\f{1}{2}}(y))}\leq C(1+||u||_{L^{\infty}(\Rn)})^q 
\quad\text{when}\quad s=\f{1}{2},\ee
Note that in \eqref{18} and \eqref{19} constants $C$ are same as in \eqref{16} and \eqref{17} respectively. Thus, in \eqref{18} and \eqref{19} constants do not depend on $y$. This implies $u\in C^{2s}(\Rn)$ when $s\not=\f{1}{2}$ and in $C^{2s-\eps}(\Rn)$, when $s=\f{1}{2}$. Hence, $f(u)\in C^{2s}(\Rn)$ when  $s\not=\f{1}{2}$ and in $C^{2s-\eps}(\Rn)$, when $s=\f{1}{2}$. Therefore, applying Theorem \ref{t:RS}(b), we have
\bea\lab{20}
||u||_{C^{4s}(B_{\f{1}{2}}(0))}&\leq& C(||u||_{C^{2s}(\Rn)}+||f(u)||_{C^{2s}(B_1(0))}) \no\\
&\leq& C(||u||_{C^{2s}(\Rn)}+||f(u)||_{C^{2s}(\Rn)}) \no\\
&\leq& C(1+||u||_{L^{\infty}(\Rn)})^{2q}\quad\text{if}\quad s\not=\f{1}{4}, \f{1}{2},  \f{3}{4}.
\eea
Similarly,
\bea\lab{21}
||u||_{C^{4s-\eps}(B_{\f{1}{2}}(0))}&\leq& C(||u||_{C^{2s-\eps}(\Rn)}+||f(u)||_{C^{2s-\eps}(B_1(0))})\no\\
&\leq& C (1+||u||_{L^{\infty}(\Rn)})^{2q} \quad\text{if}\quad s=\f{1}{2} \ \text{and}\ 4s-\eps\not\in\N.  
\eea
Arguing as before, we can show that $u\in C^{4s}(\Rn)$ when $s\not=\f{1}{2}$ and in $C^{4s-\eps}(\Rn)$, when $s=\f{1}{2}$. We can repeat this argument to improve the regularity $C^{\infty}(\Rn)$  if both $p$ and $q$ are integer and $C^{2ks+2s}(\Rn)$, where $k$ is the largest integer satisfying   $\lfloor 2ks\rfloor<p$  if $p\not\in\N$ and $\lfloor 2ks\rfloor<q$ if $p\in\N$ but $q\not\in\N$, where $\lfloor 2ks\rfloor$ denotes the greatest integer less than equal to $2ks$ .

\vspace{2mm}

(ii) Suppose, $u$ is an arbitrary solution of \eqref{eq:a3'}, then by Theorem \ref{t:est}, $u\in L^{\infty}(\Rn)$ and thus $f(u)=u^p-u^q\in L^{\infty}(\Rn)$. Consequently, by \cite[Proposition 1.1]{RS4}, it follows $u\in C^s(\Rn)$. Since $q, p>1$, we have $f(u)\in C^s_{loc}(\Rn)$. Therefore by Theorem \ref{t:RS}(ii), $u\in C^{2s+\al}_{loc}(\Om)$ for some $\al\in(0,1)$.
 
\end{proof}

\begin{proposition}\lab{p:class} Let $p,\ q,\ s$ are as in Theorem \ref{t:est}. 
If $u$ is any nonnegative weak solution of Eq.\eqref{entire} or \eqref{eq:a3'},  then $u$ is a classical solution.\end{proposition}

 \begin{proof}
{\bf Case 1}: Let $u$ be a weak solution of \eqref{entire}. 

  First, we show that $(-\De)^s u(x)$ can be defined as in \eqref{De-u}. Using $u\in L^{\infty}(\Rn)$, we see that
\begin{equation*}
\bigg|\int_{\Rn\setminus B_{\f{1}{2}}(0)}\f{u(x+y)-2u(x)+u(x-y)}{|y|^{N+2s}}dy\bigg|\leq C\int_{\Rn\setminus B_{\f{1}{2}}(0)}\f{dy}{|y|^{N+2s}}<\infty.
\end{equation*}
On the other hand,   since by Theorem \ref{t:reg-1}, $u\in C^{2s+\al}_{loc}(\Rn)$ for some $\al\in (0,1)$, it follows that $\bigg|\displaystyle\int_{B_{\f{1}{2}}(0)}\f{u(x+y)-2u(x)+u(x-y)}{|y|^{N+2s}}dy\bigg|<\infty$. Hence $(-\De)^s u(x)$ is defined pointwise.

Next, we show that  the Eq. \eqref{entire} is satisfied in pointwise sense.  $u$ is a weak solution implies
$$\int_{\Rn}(-\De)^{\f{s}{2}}u(-\De)^{\f{s}{2}}\va\ dx=\int_{\Rn}u^p\va\ dx-\int_{\Rn}u^q\va\ dx \quad\forall\va\in C^{\infty}_0(\Rn).$$
This in turn implies
$$\int_{\Rn}\va(-\De)^{s}u\ dx=\int_{\Rn}u^p\va\ dx-\int_{\Rn}u^q\va\ dx \quad\forall\va\in C^{\infty}_0(\Rn).$$ Therefore, $(-\De)^s u=u^p-u^q$ in $\Rn$  almost everywhere and $u\in C^{2s+\al}$ implies $$(-\De)^s u(x)=u^p(x)-u^q(x)\quad\forall x \in\Rn.$$
Hence, $u$ is a classical solution of \eqref{entire}.

\vspace{2mm}

{\bf Case 2:} Suppose $u$ is a weak solution of \eqref{eq:a3'}. Then applying Theorem \ref{t:est} and Theorem \ref{t:reg-1}, we can show as in Case 1 that   $(-\De)^s u(x)$ can be defined in pointwise sense.

Now we are left to show that \eqref{eq:a3'} is satisfied in pointwise sense. Towards this goal, we
define 
$$f(u)=u^p-u^q, \quad u_{\eps}:=u*\rho_{\eps} \quad \text{and}\quad f_{\eps}:=f(u)*\rho_{\eps},$$ where $\rho_{\eps}$ is the standard molifier. 
Namely, we take $\rho_{\eps}=\eps^{-N}\rho(\f{x}{\eps})$ where\\
 $\rho\in C^{\infty}_0(\Rn)$ with $0\leq\rho\leq 1$, supp $\rho\subseteq\{|x|\leq 1\}$ and $\displaystyle\int_{\Rn}\rho\ dx=1$.

Then $u_{\eps}, \  f_{\eps}\in C^{\infty}$. Proceeding along the same line as in the proof of \cite[Proposition 5]{SV3}, we can show that, for $\eps>0$ small enough it holds  
\be\lab{3}
(-\De)^s u_{\eps}=f_{\eps} \quad \text{in}\quad U,\ee
 in the classical sense, where $U$ is any arbitrary subset of $\Om$ with $U\subset\subset\Om$. Moreover, it is easy to note that
$u_{\eps}\to u$ and $f_{\eps}\to f(u)$ locally uniformly and 
$$||u_{\eps}||_{L^{\infty}(B_1(0))}\leq ||u||_{L^{\infty}(\Rn)} \quad \text{and}\quad ||f_{\eps}||_{L^{\infty}(B_1(0))}\leq C||u||_{L^{\infty}(\Rn)}.$$ 
Taking the limit $\eps\to 0$ on both the sides of \eqref{3} and using the regularity estimate of $u_{\eps}$ from theorem \ref{t:reg-1}, we obtain, $$\lim_{\eps\to 0}\int_{\Rn}\f{u_{\eps}(x+y)-2u_{\eps}(x)+u_{\eps}(x-y)}{|y|^{N+2s}}dy=f(u).$$
Using the arguments used before, it is not difficult to check that LHS of above relation converges to  $(-\De)^s u$ as $\eps\to 0$ and hence the result follows.

 \end{proof}

{\bf Proof of Theorem \ref{t:grad-est}}\begin{proof}
First, we observe that  from Theorem \ref{t:reg-1}, it follows $u\in C^1(\Rn)$. Let $R_0$ be as in Theorem \ref{t:est}. 
For $R>R_0$,  define $v(x)=R^{N-2s}u(Rx)$. Then 
\bea\lab{oct-6-1}
(-\De)^s v(x)&=&R^{N}\big((-\De)^s u\big)(Rx)\no\\
&=&R^N(u^p(Rx)-u^q(Rx))\no\\
&=&R^{N-p(N-2s)}v^p-R^{N-q(N-2s)}v^q.
\eea
From Theorem \ref{t:est}, we have $|u(x)|\leq \frac{C}{|x|^{N-2s}}$ for $|x|>R_0$. Consequently, we get 
\be\lab{oct-6-2}|v(x)|\leq \f{C}{|x|^{N-2s}} \quad\text{for}\quad |x|>\frac{R_0}{R},\ee
where $C$ is independent of $R$.  
Since $q>p\geq \f{N+2s}{N-2s}$, it follows $N-q(N-2s)<N-p(N-2s)<0$ and thus $(-\De)^s v\in L^{\infty}(B_{\f{R_0}{R}}(0))^c$ and that $L^{\infty}$ bound does not depend on $R$. \\
Let $A_1:=\{1<|x|<2 \}$ and $x_0\in A_1$. Suppose $r>0$ is such that $B_{2r}(x_0)\subset A_1$.  We choose $\eta\in C^{\infty}_0(\Rn)$ such that $\eta=1$ in $B_r(x_0)$ and supp $\eta\subset B_{2r}(x_0)$. Clearly $v\eta\in L^{\infty}(\Rn)$ and $||\eta v||_{L^{\infty}(\Rn)}\leq C_1$, where $C_1$ is independent of $R$. Moreover,
\be\lab{oct-7-3}(-\De)^s(v\eta)=(-\De)^s v+(-\De)^s\big((\eta-1)v\big).\ee
Note that, for  $z\in B_r(x_0)$ we have 
$$(-\De)^s\big((\eta-1)v\big)(z)=c_{N,s}\displaystyle\int_{\Rn\setminus B_r(x_0)}\frac{-\big((\eta-1)v\big)(y)}{|z-y|^{N+2s}}dy.$$
From this expression we obtain
\bea\lab{oct-6-3}
||(-\De)^s\big((\eta-1)v\big)||_{L^{\infty}(B_r(x_0))}&\leq& C\displaystyle\int_{\Rn}\frac{v(y)}{(1+|y|)^{N+2s}}dy\no\\
&=&C\int_{B_\f{R_0}{R}(0)} \frac{v(y)}{(1+|y|)^{N+2s}}dy\no\\
&+&C\int_{|y|>\f{R_0}{R}} \frac{v(y)}{(1+|y|)^{N+2s}}dy.
\eea
Now, using the definition of $v$ and the fact that $u\in L^{\infty}(\Rn)$, we get
\bea\lab{oct-7-1}
\int_{B_\f{R_0}{R}(0)} \frac{v(y)}{(1+|y|)^{N+2s}}dy&=&R^{N-2s}\int_{B_\f{R_0}{R}(0)} \frac{u(Ry)}{(1+|y|)^{N+2s}}dy\no\\
&=&CR^N\int_{B_{R_0}(0)} \frac{u(x)dx}{(R+|x|)^{N+2s}}\no\\
&\leq&C\f{R^N}{R^{N+2s}}|B_{R_0}(0)|<C',
\eea
where $C'$ is independent of $R$ (since, $R^{-2s}<1$).
On the other hand, using \eqref{oct-6-2} we have
\bea\lab{oct-7-2}
\int_{|y|>\f{R_0}{R}} \frac{v(y)}{(1+|y|)^{N+2s}}dy&=&C\int_{|y|>\f{R_0}{R}}\f{dy}{|y|^{N-2s}(1+|y|)^{N+2s}}\no\\
&\leq&C\int_{\Rn}\f{dy}{|y|^{N-2s}(1+|y|)^{N+2s}}\no\\
&\leq&C\int_{B_1(0)}\f{dy}{|y|^{N-2s}}+\int_{|y|>1}\f{dy}{|y|^{2N}}\no\\
&\leq&C,
\eea
for some constant $C>0$, which does not depend on $R$. Plugging \eqref{oct-7-1} and \eqref{oct-7-2} into \eqref{oct-6-3} and then using \eqref{oct-7-3} we obtain $||(-\De)^s(\eta v)||_{L^{\infty}(B_{r}(x_0))}<C$, where $C$  depends only on $N, s, p, q, R_0$.
Consequently, using \cite[Proposition 2.3]{RS4}, we get 
$$||(\eta v)||_{C^{\ba}(\overline{B_\f{r}{2}(x_0)})}\leq C \quad\forall\ \ba\in(0,2s),$$
where $C$  depends only on $N, s, p, q, R_0$. As a consequence,
$$||v||_{C^{\ba}(\overline{B_\f{r}{2}(x_0)})}\leq C.$$
Thus, thanks to \cite[Corollary 2.4]{RS4} we have $$||v||_{C^{\ba+2s}(\overline{B_\f{r}{8}(x_0)})}\leq C.$$
We continue to apply this bootstrap argument and after a finitely many steps we have $||v||_{C^{\ba+ks}(\overline{B_{r_0}(x_0)})}\leq C$.
 for some $r_0>0$ and $\ba+ks>1$. This in turn implies 
$||\na v||_{L^{\infty}(\overline{B_{r_0}(x_0)})}\leq C$. This further yields to
$$||\na v||_{L^{\infty}(A_1)}\leq C,$$
 where $C$ depends only on $N, s, p, q, R_0$.  Therefore, using the definition of $v$, we obtain 
 $$|\na u(Rx)|\leq\f{C}{R^{N-2s+1}} \quad\text{for}\ 1<|x|<2.$$ From the above expression, it is easy to deduce that
 $$|\na u(y)|\leq\f{C}{|y|^{N-2s+1}} \quad\text{for}\ R<|y|<2R.$$
As $R>R_0$ was arbitrary we get 
 $$|\na u(y)|\leq\f{C}{|y|^{N-2s+1}} \quad\text{for}\ |y|>R,$$ for some $R$ large.

\end{proof} 

\section{Pohozaev identity and nonexistence result}
{\bf Proof of Theorem \ref{t:poho}:}
We prove this theorem by establishing Pohozaev identity in the spirit of Ros-Oton and Serra \cite{RS2}.
For $\la>0$, define $u_{\la}(x)=u(\la x)$. Multiplying the equation \eqref{entire} by $u_{\la}$ yields,
\bea\lab{oct-7-4}
\int_{\Rn}(u^p-u^q)u_{\la}dx&=&\int_{\Rn}(-\De)^\f{s}{2}u(-\De)^\f{s}{2}u_{\la}dx\no\\
&=&\la^s\int_{\Rn}(-\De)^\f{s}{2}u(x)\big((-\De)^\f{s}{2}u\big)(\la x)dx\no\\
&=&\la^s\int_{\Rn}ww_{\la}dx,
\eea
where, $w(x):=(-\De)^\f{s}{2}u(x)$ and $w_{\la}(x)=w(\la x)$. 
With the change of variable $x=\sqrt{\la}y$, we have
\be\lab{oct-7-5}
\la^s\int_{\Rn}ww_{\la}dx=\la^s\int_{\Rn}w(x)w(\la x)dx
=\la^{-\f{N-2s}{2}}\int_{\Rn}w_{\sqrt{\la}}w_\f{1}{\sqrt{\la}} dy.
\ee
Therefore, \be\lab{oct-7-8}
\int_{\Rn}(u^p-u^q)u_{\la}dx=\la^{-\f{N-2s}{2}}\int_{\Rn}w_{\sqrt{\la}}w_\f{1}{\sqrt{\la}} dy.
\ee
Observe that using the decay estimate at infinity of $u$ and $\na u$ from Theorem \ref{t:est} and Theorem \ref{t:grad-est} , we get $\displaystyle\int_{\Rn}(u^p-u^q)(x\cdot\na u)dx$ is well defined and that integral can be written as $\displaystyle\int_{\Rn}x\cdot\na\bigg(\f{u^{p+1}}{p+1}-\f{u^{q+1}}{q+1}\bigg)dx.$
Again using the decay estimate of $u$ from Theorem \ref{t:est}, we justify the following integration by parts 
\be\lab{oct-7-6}-\f{N}{p+1}\int_{\Rn}u^{p+1}dx+\f{N}{q+1}\int_{\Rn}u^{q+1}dx=\int_{\Rn}x\cdot\na\bigg(\f{u^{p+1}}{p+1}-\f{u^{q+1}}{q+1}\bigg)dx.\ee
Thus, using \eqref{oct-7-8} we simplify the LHS of above expression as follows:
\bea\lab{oct-7-7}
\text{LHS of \eqref{oct-7-6}}&=&\displaystyle\int_{\Rn}(u^p-u^q)(x\cdot\na u)dx\no\\
&=&\f{d}{d\la}\bigg|_{\la=1}\int_{\Rn}(u^p-u^q)u_{\la} dx\no\\
&=&\f{d}{d\la}\bigg|_{\la=1}\bigg(\la^{-\f{N-2s}{2}}\int_{\Rn}w_{\sqrt{\la}}w_\f{1}{\sqrt{\la}}\bigg) dx.\no\\
&=&-\bigg(\f{N-2s}{2}\bigg)\int_{\Rn}w^2 dx+\f{d}{d\la}\bigg|_{\la=1}\int_{\Rn}w_{\sqrt{\la}}w_\f{1}{\sqrt{\la}}dy\no\\
&=&-\bigg(\f{N-2s}{2}\bigg)||u||_{\dot{H}^s(\Rn)}^2.
\eea

On the other hand, multiplying \eqref{entire} by $u$ we have,
$$||u||_{\dot{H}^s(\Rn)}^2=\int_{\Rn}(u^{p+1}-u^{q+1})dx.$$
Combining this expression along with \eqref{oct-7-7} we obtain the Pohozaev identity
$$\bigg(\f{N-2s}{2}-\f{N}{p+1}\bigg)\int_{\Rn}u^{p+1}dx=\bigg(\f{N-2s}{2}-\f{N}{q+1}\bigg)\int_{\Rn}u^{q+1}dx.$$
Clearly, from the above identity, it follows that \eqref{entire} does not admit any solution when $p=2^*-1$ and $q>p$. This completes the theorem.
\hfill$\square$

\section{Symmetry and monotonically decreasing property}
\begin{theorem}
Let $p, q, s$ are as in Theorem \ref{t:est}  and $u$ be any solution of Eq.\eqref{entire}. Then $u$ is radially symmetric and strictly decreasing about some point in $\Rn$.
\end{theorem}
\begin{proof}
By Proposition \ref{p:class}, $u$ is a classical solution of \eqref{entire}. Define $f(u)=u^p-u^q$. Then clearly $f$ is locally Lipschitz. \\

{\bf Claim:} There exists $s_0, \ga, C>0 $ such that 
$$\f{f(v)-f(u)}{v-u}\leq C(u+v)^{\ga} \quad\text{for all}\quad 0<u<v<s_0.$$

To see the claim,
\bea
f(v)-f(u) &=&(v^p-u^p)-(v^q-u^q)\no\\
&=&p\big(\theta_1 v+(1-\theta_1)u\big)^{p-1}(v-u)-q\big(\theta_2 v+(1-\theta_2)u\big)^{q-1}(v-u),\no
\eea
for some $\theta_1, \theta_2\in(0,1)$.  Thus, for $0<u<v$
\bea
\f{f(v)-f(u)}{v-u}&=&p\big(\theta_1 v+(1-\theta_1)u\big)^{p-1}-q\big(\theta_2 v+(1-\theta_2)u\big)^{q-1}\no\\
&\leq&p\big(\theta_1 v+(1-\theta_1)u\big)^{p-1}\no\\
&\leq& p(u+v)^{p-1}.\no
\eea
Therefore, the claim holds with $C=p$ and $\ga=p-1$ and for any positive $s_0$. 

Moreover, from Theorem \ref{t:reg-1}, we have 
$$u(x)=O(\f{1}{|x|^{N-2s}}) \quad\text{as}\quad |x|\to\infty.$$

Since $p\geq \f{N+2s}{N-2s}$, it is easy to check that $$N-2s>\max\bigg(\f{2s}{\ga}, \f{N}{\ga+2}\bigg),$$
where $\ga=p-1$, as found in the above claim.  Hence, the theorem follows from \cite[Theorem 1.2]{FW1}.
\end{proof}

\begin{theorem}
Suppose $\Om$ is a smooth bounded convex domain, $p, q, s$ are as in Theorem \ref{t:est} . Assume further that 
$\Om$ is convex in $x_1$ direction and symmetric w.r.t. to the hyperplane $x_1=0$. Let $s\in(0,1)$ and $u$ be any  solution of Eq.\eqref{eq:a3'}. Then $u$ is symmetric w.r.t. $x_1$ and strictly decreasing in $x_1$ direction for $x=(x_1,x')\in\Om$, $x_1>0$.
\end{theorem}
\begin{proof}
Follows from \cite[Theorem 3.1]{FW} (also see \cite[Cor. 1.2]{JW}).
\end{proof}

\section{Existence results}
\begin{lemma}\lab{rad-est}
Let $s\in(0,1)$. If $u$ is any radially symmetric decreasing function in $\dot{H}^s(\Rn)$, then
$$u(|x|)\leq \f{C}{|x|^\f{N-2s}{2}}.$$
\end{lemma}
\begin{proof}
It is enough to show that if $u\in \dot{H}^s(\Rn)$ with $u(x)=u(|x|)$ and \\ $u(r_1)\leq u(r_2)$, when $r_1\geq r_2$, then it holds $u(R)\leq \f{C}{R^\f{N-2s}{2}}$ for any $R>0$. To see this, we note that by Sobolev inequality we can write,
\bea
\f{1}{S}||(-\De)^{\f{s}{2}}u||_{L^2(\Rn)} &\geq& \bigg(\int_{\Rn} |u(x)|^{2^*}dx\bigg)^{\f{1}{2^*}}\no\\
&\geq&\bigg(\int_{0}^R\int_{\pa B_r} |u(r)|^{2^*} dS dr\bigg)^{\f{1}{2^*}}\no\\
&\geq& u(R)\bigg(\int_0^R \om_n r^{N-1}dr\bigg)^{\f{1}{2^*}}\no\\
&=&\bigg(\f{\om_N}{N}\bigg)^\f{1}{2^*}u(R) R^{\f{N}{2^*}}.
\eea
As $u\in \dot{H}^s(\Rn)$ implies LHS is bounded above, the above inequality yields
 $$u(R)\leq \bigg(\f{N}{\om_N}\bigg)^\f{1}{2^*}\f{1}{S}||(-\De)^{\f{s}{2}}u||_{L^2(\Rn)} R^{-\f{N-2s}{2}}\leq C R^{-\f{N-2s}{2}}.$$
\end{proof}

{\bf Proof of Theorem \ref{1.2}}\begin{proof}
We are going to work on the manifold $$\mathcal{N}= \bigg\{ u\in \dot{H}^s(\R^N) \cap L^{q+1}(\R^N):  \int_{\R^N} |u|^{p+1}dx=1 \bigg\},$$ and  $F(.)$ on $\mathcal{N}$ reduces  as
$$F(u)=  \frac{1}{2}\int_{\R^N}\int_{\Rn}\frac{|u(x)-u(y)|^2}{|x-y|^{N+2s}} dxdy +  \frac{1}{q+1}\int_{\R^N}|u|^{q+1} dx.$$

Let $u_{n} $ be a minimizing sequence in $\mathcal{N}$ such that
$$F(u_{n}) \to \mathcal{K} \text{ with } \int_{\R^N} |u_{n}|^{p+1} dx =1.$$
Thus, $\{u_n\}$ is a bounded sequence in $\dot{H}^s(\Rn)$ and $L^{q+1}(\Rn)$. Therefore, there exists $u\in \dot{H}^s(\Rn)$ and $L^{q+1}(\Rn)$ such that $u_n \rightharpoonup u$ in $\dot{H}^s(\Rn)$ and $L^{q+1}(\Rn)$. Consequently $u_{n}\to u$ pointwise almost everywhere.

Using symmetric rearrangement technique, without loss of generality, we can assume that $u_n$ is radially symmetric and decreasing (see \cite{P}). 
We claim that $u_{n} \to u$ in $L^{p+1}(\R^N).$\\
To see the claim, we note that $u_{n}^{p+1}\to u^{p+1}$ pointwise almost everywhere.  Since $\{u_n\}$ is uniformly bounded in $L^{q+1}(\Rn)$, using Vitali's convergence theorem, it is easy to check that $\displaystyle\int_{K} |u_{n}|^{p+1}dx\to \int_{K}  |u|^{p+1}dx$  for any compact set $K$  in $\R^N$ containing the origin.
Furthermore, applying Lemma \ref{rad-est} it follows, $\displaystyle\int_{\R^N \setminus K}|u_{n}|^{p+1} dx$ is very small  and hence we have strong convergence. Moreover, $\displaystyle\int_{\R^N}|u_{n}|^{p+1}dx=1 $  implies
$\displaystyle\int_{\R^N }|u|^{p+1} dx=1.$\\

Now we show that $\mathcal{K}= F(u).$

We note that $u\mapsto ||u||^2$ is weakly lower semicontinuous. Using this fact along with  Fatou's lemma, we have

\begin{eqnarray*}\mathcal{K} &=&   \lim_{n \to \infty } \bigg[\frac{1}{2}\int_{\R^N}\int_{\Rn}\f{|u_{n}(x)-u_n(y)|^2}{|x-y|^{N+2s}} dxdy+  \frac{1}{q+1}\int_{\R^N}|u_{n}|^{q+1} dx\bigg] \\
&=&\lim_{n \to \infty } \bigg[\frac{1}{2}||u_n||^2+\frac{1}{q+1}\int_{\R^N}|u_{n}|^{q+1} dx\bigg]\\
&\geq & \frac{1}{2}||u||^2+  \frac{1}{q+1}\int_{\R^N}|u|^{q+1} dx\bigg]\\
&\geq & F(u)   .\end{eqnarray*}
This proves $F(u)=\mathcal{K}$. Moreover, using the  symmetric rearrangement technique via. Polya-Szego inequality (see \cite{P}), it is easy to check that $u$ is nonnegative, radially symmetric and radially  decreasing 
 Applying the Lagrange multiplier rule, we obtain $u$ satisfies
\be\no - \De u+ u^q=  \lambda u^p, \ee
for some $\lambda>0$. This in turn implies
\be\no (- \De)^s u=  \lambda u^p - u^q  ~~~~\text{ in } ~~~ \R^N.\ee
Finally, if $q>(p-1)\f{N}{2s}-1$, then we know that $u$ is a classical solution. Therefore, if there exists $x_0\in\Rn$ such that $u(x_0)=0$, that that would imply $(-\De)^s u(x_0)<0$ (since, $u$ is a nontrivial solution). On the other hand, $(\lambda u^p - u^q)(x_0)=0$ and that yields a contradiction. Hence $u>0$ in $\Rn$.

Furthermore, we observe that by setting $v(x)=\la^{-\f{1}{q-p}}u(\la^{-\f{q-1}{2s(q-p)}}x)$, it holds 
$$(-\De)^s v=v^p-v^q \quad\text{in}\quad\Rn.$$ Hence the theorem follows.

\end{proof}

{\bf Proof of Theorem \ref{1.3}}\begin{proof}
We are going to work on the manifold $$\mathcal{\tilde N}= \bigg\{ u\in X_0 \cap L^{q+1}(\Om): \int_{\Om} |u|^{p+1}=1 \bigg\}.$$
Then $F_{\Om}$ reduces to
$$F_{\Om}(u)=  \frac{1}{2}\int_{\R^N}\int_{\Rn}\frac{|u(x)-u(y)|^2}{|x-y|^{N+2s}} dxdy +  \frac{1}{q+1}\int_{\Om}|u|^{q+1} dx . $$
Let $u_{n} $ be a minimizing sequence in $\mathcal{\tilde N}$ such that $F_{\Om}(u_{n}) \to S_{\Om},$ then 
$$F(u_{n}) \to S_{\Om} \text{ with } \int_{\Om} |u_{n}|^{p+1} dx =1.$$
Then $u_{n}$ is bounded in $X_0 \cap L^{q+1}(\Om).$ Consequently,  $u_{n}\rightharpoonup u$ on $H^s(\Om)$ and $u_{n} \to u$ on $L^2(\Om).$ As a result,  $u_{n}\to u$ pointwise almost everywhere. By the interpolation inequality,  we must have $u_{n} \to u$ on $ L^{p+1}(\Om).$
Hence, $\displaystyle\int_{\Om }|u|^{p+1} dx =1.$\\

Now we show that $S_{\Om}= F_{\Om}(u).$ Using Fatou's Lemma and the fact that $u\mapsto ||u||^2$ is weakly lower semicontinuous ,
\begin{eqnarray*}S_{\Om} &=&   \lim_{n \to \infty } \bigg[\frac{1}{2}\int_{\R^N}\int_{\Rn}\frac{|u_n(x)-u_n(y)|^2}{|x-y|^{N+2s}} dxdy +  \frac{1}{q+1}\int_{\Om}|u_{n}|^{q+1} dx\bigg] \\&\geq &  \bigg[\frac{1}{2}\int_{\R^N}\int_{\Rn}\frac{|u(x)-u(y)|^2}{|x-y|^{N+2s}} dxdy+  \frac{1 }{q+1}\int_{\Om}|u|^{q+1} dx\bigg]\\ &\geq & F_{\Om}(u)   .\end{eqnarray*} 

By the Lagrange multiplier rule, we obtain $u$ satisfies
\be\no (- \De)^s u+  |u|^{q-1}u=  \lambda |u|^{p-1}u. \ee

Now we replace $\mathcal{\tilde N}$ by $\mathcal{\tilde N}_+:= \bigg\{ u\in X_0 \cap L^{q+1}(\Om): \displaystyle\int_{\Om} (u^+)^{p+1}=1 \bigg\}$, the functional $F_{\Om}(.)$ by $\tilde{F}_{\Om}(.)$ defined as follows
$$\tilde{F}_{\Om}(u):=  \frac{1}{2}\int_{\R^N}\int_{\Rn}\frac{|u(x)-u(y)|^2}{|x-y|^{N+2s}} dxdy +  \frac{1}{q+1}\int_{\Om}(u^+)^{q+1} dx ,$$ and 
$S_{\Om}$ by $\tilde{S}_{\Om}:= \inf\bigg\{F(v, \Om): v\in \mathcal{\tilde N}_+\bigg\}$. Repeating the same argument as before (with a little modification), it can be easily shown that there exists $u\in X_0\cap L^{q+1}(\Om)$ which satisfies
\be\lab{Mar-18-1}
(- \De)^s u+  (u^+)^q=  \lambda (u^+)^p \quad\text{in}\quad\Om.
\ee
Taking $u^-$ as the test function for \eqref{Mar-18-1} we obtain from Definition \ref{def-1} that 
\be\lab{Mar-18-2}\int_{\Rn}\int_{\Rn}\f{(u(x)-u(y))(u^-(x)-u^-(y))}{|x-y|^{N+2s}}dxdy=0.\ee
Furthermore,
\Bea
\text{LHS of}\  \eqref{Mar-18-2} &=&\int_{\Rn}\int_{\Rn}\f{(u(x)-u(y))(u^-(x)-u^-(y))}{|x-y|^{N+2s}}dxdy\\
&=&\int_{\Rn}\int_{\Rn}\f{\big((u^+(x)-u^+(y))-(u^-(x)-u^-(y))\big)(u^-(x)-u^-(y))}{|x-y|^{N+2s}}dxdy\\
&=&-u^-(x)u^+(y)-u^+(x)u^-(y)-||u^-||^2\\
&\leq&-||u^-||^2
\Eea
Hence, from \eqref{Mar-18-2} we obtain $u^-=0$, i.e, $u\geq 0$. Moreover, since for $p\geq 2^*-1$ and $q>(p-1)\f{N}{2s}-1$,  Proposition \ref{p:class} implies $u$ is a classical solution,  applying maximum principle as in Theorem \ref{1.2}, we conclude $u>0$ in $\Om$. This completes the proof.
\end{proof}

\appendix
\section{}
In this section we give an alternative proof of Pohozaev identity in $\Rn$ for the following type of equations:
\be\lab{oct-7-9}(-\De)^s u=f(u) \quad\text{in}\quad\Rn,\ee
 where $u\in \dot{H}^s(\Rn)\cap L^{\infty}(\Rn)$ and $f\in C^2$. Here we do not require the decay estimate of $u$ or $\na u$ at infinity.
\begin{theorem}
Let $u\in \dot{H}^s(\Rn)\cap L^{\infty}(\Rn)$ be a  positive solution of \eqref{oct-7-9} and $F(u)\in L^1(\Rn)$. 
Then
$$(N-2s)\int_{\Rn}uf(u)\ dx=2N\int_{\Rn}F(u)\ dx,$$
where $F(u)=\displaystyle\int_{0}^u f(t)dt$.
\end{theorem}

\begin{proof}
We prove this theorem using the harmonic extension method introduced in Section 2. 

Let $u$ be a nontrivial positive solution of \eqref{oct-7-9}. Suppose, $w$ is the harmonic extension of $u$. Then $w$ is a solution of
\begin{equation}\lab{A-42'}
\left\{\begin{aligned}
      \text{div}(y^{1-2s} \na w) &=0  \quad\text{in}\quad \R^{N+1}_+,\\
  \frac{\pa w}{\pa \nu ^{2s} } &=f(w(.,0)) \quad\text{on}\quad\mathbb{R}^N,
          \end{aligned}
  \right.
\end{equation}
(see \eqref{A-42}).
For $r>0$, we define $B_r$ to be the ball in $\R^{N+1}$, that is, 
$$B_r:=\{(x,y)\in\R^{N+1}: |(x,y)|<r\}.$$
Define $$B_r^{+}=B_r\cap\R^{N+1}_+$$ and 
$$Q_r=B_r^+\cup(B_r\cap(\Rn\times\{0\})).$$
Let $\va\in C_{0}^{\infty}(\R^{N+1})$ with $0\leq\va\leq 1$, 
$\va=1$ in $B_1$, $\va$ has support in $B_2$ and $|\na\va|\leq 2$. For $R>0$, define
$$\psi_R(x,y)=\psi\bigg(\f{(x,y)}{R}\bigg), \quad\text{where}\quad \psi=\va|_{\overline{\R^{N+1}_+}}.$$
Multiplying \eqref{A-42'} by $((x,y)\cdot\na w)\psi_R$ and integrating in $\R^{N+1}_+$ we have,
\be\lab{sep-27-1}
\displaystyle\int_{Q_{2R}}\text{div}(y^{1-2s}\na w)\big[((x,y)\cdot\na w)\psi_R\big] dxdy=0.
\ee
Then integration by parts yields
\bea\lab{sep-27-2}
&\displaystyle\int_{Q_{2R}} y^{1-2s}\na w\na \big[((x,y)\cdot\na w)\psi_R\big] dxdy\no\\
&=\displaystyle\int_{\pa Q_{2R}} y^{1-2s}(\na w\cdot {\bf n}) \big[((x,y)\cdot\na w)\psi_R\big] dS\no\\
&\qquad\qquad\qquad\qquad=-\lim_{y\to 0+}\displaystyle\int_{B_{2R}\cap(\Rn\times\{y\})}y^{1-2s}\f{\pa w}{\pa y}(x,y)\big((x,y)\cdot\na w\big) \psi_R dx\no\\
&=k_{2s}^{-1}\displaystyle\int_{B_{2R}\cap(\Rn\times\{0\})}\big(x\cdot\na_x w\big) \psi_R \f{\pa w}{\pa\nu^{2s}}dx
\eea
where $k_{2s}$ is as defined in Section 2. In the above steps we have used the fact that $\psi_R=0$ on $\pa B_{2R}$. From \eqref{A-42'}, we know    $ \f{\pa w}{\pa\nu^{2s}}=f(w(x,0))$ on $\Rn$.  Therefore, RHS of  \eqref{sep-27-2} simplifies as
\bea\lab{sep-27-3}
\text{RHS of \eqref{sep-27-2}}&=&k_{2s}^{-1}\int_{B_{2R}\cap(\Rn\times\{0\})}\big(x\cdot\na_x w\big)f(w)\psi_R dx\no\\
&=& k_{2s}^{-1}\int_{B_{2R}\cap(\Rn\times\{0\})}\big(x\cdot\na_x F(w)\big)\psi_R dx\no\\
&=&-Nk_{2s}^{-1}\int_{B_{2R}\cap(\Rn\times\{0\})}F(w)\psi_R dx\no\\
&-&k_{2s}^{-1}\int_{B_{2R}\cap(\Rn\times\{0\})}F(w)(x\cdot\na_x\psi_R)dx.
\eea
 Since $w(x,0)=u(x)$, from \eqref{A-42'}, we find $F(w)=F(u)$ on $\Rn$. Moreover, $|\na\psi_R|\leq \f{2}{R}$. Hence, the 2nd integral on RHS of \eqref{sep-27-3} can be written as 
\bea\lab{sep-27-4}
\int_{B_{2R}\cap(\Rn\times\{0\})}F(w)(x\cdot\na_x\psi_R)dx&\leq& C\int_{(B_{2R}\setminus B_{R})\cap(\Rn\times\{0\})}F(u)\f{|x|}{R}dx\no\\
&\leq& C\int_{(B_{2R}\setminus B_{R})\cap(\Rn\times\{0\})}F(u)dx,
\eea
which converges to $0$ as $R\to\infty$ (since, $F(u)\in L^{1}(\Rn)$).
As a result, 
\bea\lab{sep-27-5}
\lim_{R\to\infty}\text{RHS of \eqref{sep-27-2}}=-Nk_{2s}^{-1}\int_{\Rn}F(u)dx
\eea
Next, we like to simplify LHS of \eqref{sep-27-2}. Towards this aim,
let us first simplify the term $\na w\na \big[((x,y)\cdot\na w)\psi_R\big]$.
\be\lab{sep-27-6}
\na w\na \big[((x,y)\cdot\na w)\psi_R\big]=(\na w\cdot\na\psi_R)((x,y)\cdot\na w)+\na w\cdot \na ((x,y)\cdot\na w)\ \psi_R.
\ee
By doing  a straight forward computation, we further simplify the 2nd term on the RHS of above expression as below:
\be
\na w\cdot \na ((x,y)\cdot\na w)\ \psi_R=\big[|\na w|^2+\f{1}{2}\sum_{j=1}^N\f{\pa}{\pa x_j}(|\na w|^2)x_j+\f{1}{2}\f{\pa}{\pa y}(|\na w|^2)y.\big]\psi_R.\no
\ee
Substituting back this expression into \eqref{sep-27-6} and then plugging \eqref{sep-27-6} into LHS of \eqref{sep-27-2}, we obtain
\bea\lab{sep-28-1}
\displaystyle\int_{Q_{2R}} y^{1-2s}\na w\na \big[((x,y)\cdot\na w)\psi_R\big] dxdy&=&\int_{Q_{2R}} y^{1-2s}|\na w|^2\psi_Rdxdy\no\\
&+&\f{1}{2}\int_{Q_{2R}} y^{1-2s}\bigg(\sum_{j=1}^N\f{\pa}{\pa x_j}(|\na w|^2)x_j\bigg)\psi_Rdxdy\no\\
&+&\f{1}{2}\int_{Q_{2R}} y^{2-2s}\bigg(\f{\pa}{\pa y}(|\na w|^2)\bigg)\psi_Rdxdy\no\\
&+&\int_{Q_{2R}} y^{1-2s}(\na w\cdot\na\psi_R)((x,y)\cdot\na w)dxdy.
\eea
Performing integration by parts on RHS of above expression, followed by simple computation yields,
\bea\lab{sep-28-2}
\displaystyle\int_{Q_{2R}} y^{1-2s}\na w\na \big[((x,y)\cdot\na w)\psi_R\big] dxdy&=&-\left(\f{N-2s}{2}\right)\int_{Q_{2R}} y^{1-2s}|\na w|^2\psi_Rdxdy\no\\
&-&\f{1}{2}\int_{Q_{2R}} y^{1-2s}|\na w|^2\big((x,y)\cdot\na\psi_R\big)dxdy\no\\
&+&\f{1}{2}\int_{\pa Q_{2R}}y^{1-2s}|\na w|^2\big((x,y)\cdot{\bf n}\big)\psi_RdS\no\\
&+&\int_{Q_{2R}} y^{1-2s}(\na w\cdot\na\psi_R)((x,y)\cdot\na w)dxdy.
\eea
Note that,
\bea\lab{sep-28-3}
\bigg|\int_{Q_{2R}} y^{1-2s}|\na w|^2\big((x,y)\cdot\na\psi_R\big)dxdy\bigg|&\leq&C\int_{Q_{2R}\setminus Q_R} y^{1-2s}|\na w|^2\f{|(x,y)|}{R}dxdy\no\\
&\leq&C\int_{Q_{2R}\setminus Q_R} y^{1-2s}|\na w|^2dxdy\no\\
&\To&0 \quad\text{as}\quad R\to\infty.
\eea
Similarly, \be\lab{sep-28-4}
\int_{Q_{2R}} y^{1-2s}(\na w\cdot\na\psi_R)((x,y)\cdot\na w)dxdy\To 0 \quad\text{as}\quad R\to\infty.
\ee
Since $\psi_R=0$ on $\pa B_{2R}$, we have,
\bea\lab{sep-28-5}
\f{1}{2}\int_{\pa Q_{2R}}y^{1-2s}|\na w|^2\big((x,y)\cdot{\bf n}\big)\psi_RdS&=&-\lim_{y\to0}\f{1}{2}\int_{B_{2R}\cap(\Rn\times\{0\})}y^{2-2s}|\na w|^2\psi_RdS\no\\
&=&0
\eea
Combining \eqref{sep-28-3}, \eqref{sep-28-4} and \eqref{sep-28-5} with \eqref{sep-28-2} we obtain 
\be\lab{sep-28-6}
\lim_{R\to\infty}\displaystyle\int_{Q_{2R}} y^{1-2s}\na w\na \big[((x,y)\cdot\na w)\psi_R\big] dxdy=-\left(\f{N-2s}{2}\right)\int_{\R^{N+1}_+} y^{1-2s}|\na w|^2dxdy.
\ee
Thus, \eqref{sep-28-6} and \eqref{sep-27-5} along with \eqref{sep-27-2}, yields
\be\lab{sep-28-7}
-\left(\f{N-2s}{2}\right)\int_{\R^{N+1}_+} y^{1-2s}|\na w|^2dxdy
=-Nk_{2s}^{-1}\int_{\Rn}F(u)dx.
\ee
We multiply \eqref{A-42'}  by $w\psi_R$ and then integrate by parts. Since $\psi_R=0$ on $\pa B_{2R}$ and $-\lim_{y\to 0}y^{1-2s}\f{\pa w}{\pa y}=k_{2s}^{-1}\f{\pa w}{\pa\nu^{2s}}$, we obtain
\bea\lab{sep-28-8}
\int_{Q_{2R}}y^{1-2s}\na w\cdot\na(w\psi_R)\ dxdy&=&\int_{\pa Q_{2R}}y^{1-2s}(\na w\cdot{\bf n})w\psi_R\ dS\no\\
&=&k_{2s}^{-1}\int_{B_{2R}\cap(\Rn\times\{0\})}\f{\pa w}{\pa\nu^{2s}}w\psi_R dx\no\\
&=&k_{2s}^{-1}\int_{B_{2R}\cap(\Rn\times\{0\})}f(u)u\psi_R dx.
\eea Therefore, 
\bea\lab{sep-28-9}
\lim_{R\to\infty}\text{RHS of \eqref{sep-28-8}}=k_{2s}^{-1}\int_{\Rn}uf(u)dx.
\eea
Proceeding same as before we show that 
\bea\lab{sep-28-10}
\lim_{R\to\infty}\text{LHS of \eqref{sep-28-8}}=\int_{\R^{N+1}_+}y^{1-2s}|\na w|^2\ dxdy.
\eea
Consequently, \be\lab{sep-28-11}
\int_{\R^{N+1}_+}y^{1-2s}|\na w|^2\ dxdy=k_{2s}^{-1}\int_{\Rn}f(u)udx.
\ee
Substituting \eqref{sep-28-11} into \eqref{sep-28-7}, we have
$$\left(\f{N-2s}{2}\right)\int_{\Rn}f(u)udx =N\int_{\Rn}F(u)dx,$$
which completes the proof.
\end{proof}

{\bf Acknowledgement:} Proof of Theorem \ref{t:grad-est} is based on the idea given by Dr. Joaquim Serra. The authors are indebted to him for the discussions and for his valuable comments. The first author is supported by the INSPIRE research grant
DST/INSPIRE 04/2013/000152 and the second author is supported by the
NBHM grant 2/39(12)/2014/RD-II. The authors also expresses their sincere gratitude to the anonymous referee for many valuable comments and suggestions, which helped to improve the manuscript greatly.


\begin{thebibliography}{XX}
\bibitem{A}
 {\sc D. Applebaum,}  L\'{e}vy processes from probability to finance and quantum groups, \textit{Notices Amer. Math. Soc.} 51 (2004), 1336-1347.
 
 \bibitem{BCPS}{\sc B. Barrios; E. Colorado;  A. De Pablo and U. S\'{a}nchez}, On some critical problems for the fractional
Laplacian operator. \textit{J. Diff. Eqns} 252 (2012), 6133--6162.

\bibitem{BMS}{\sc M. Bhakta; D. Mukherjee; S. Santra,} Profile of solutions for nonlocal equations with critical and supercritical nonlinearities,  \textit{submitted}, 
arXiv:1612.01759.

\bibitem{BS}{\sc M. Bhakta; S. Santra,} On a singular equation with critical and supercritical exponents, \textit{submitted}, arXiv:1608.00490 .

\bibitem{BCP}{\sc C. Br\"{a}ndle; E. Colorado;  A. de Pablo; U. S\'{a}nchez,}
A concave-convex elliptic problem involving the fractional Laplacian. \textit{Proc. Roy. Soc. Edinburgh Sect. A} 143 (2013), no. 1, 39--71.


\bibitem{CC} {\sc  X. Cabr\'{e};  E. Cinti},  Sharp energy estimates for nonlinear fractional diffusion equations. \textit{Calc. Var. Partial Differential Equations} 49 (2014), no. 1-2, 233--269.

\bibitem{CTan}{\sc X. Cabr\'{e}; J. Tan}, 
Positive solutions of nonlinear problems involving the square root of the Laplacian. 
\textit{Adv. Math.} 224 (2010), no. 5, 2052--2093. 
 
\bibitem{CS}{\sc  L. Caffarelli; L. Silvestre,} An extension problem related to the fractional Laplacian, \textit{Comm. Partial Differential Equations} 32
(2007) 1245--1260.

\bibitem{CT}
{\sc R. Cont; P. Tankov}, Financial modelling with jump processes.
 \textit{Chapman and Hall/CRC Financial Mathematics Series}, 2004.
 
 \bibitem{DDW}{\sc J. D\'{a}vila;  L. Dupaigne; J. Wei}, On the fractional Lane-Emden Equation, to appear in \textit{  Trans. Amer. Math. Soc.}
 
 \bibitem{DMV}{\sc S. Dipierro; M. Medina; E. Valdinoci}, Fractional elliptic problems with critical growth in the whole of $\Rn$. arXiv: 1506.01748.


\bibitem{FKS}{\sc E. Fabes; C. Kenig and R. Serapioni}, The local regularity of solutions of degenerate elliptic equations,
\textit{Comm. Partial Differential Equations} 7 (1982), 77--116.

 \bibitem{FW}{\sc  M. M. Fall; T. Weth},  Nonexistence results for a class of fractional elliptic boundary value problems. \textit{J. Funct. Anal.} 263 (2012), no. 8, 2205--2227.

\bibitem{FW1}{\sc  P. Felmer; Y. Wang,} Radial symmetry of positive solutions to equations involving the fractional Laplacian. \textit{Commun. Contemp. Math.} 16 (2014), no. 1, 1350023, 24 pp.

 \bibitem{GS}{\sc N. Ghoussoub; S. Shakerian}, Borderline variational problems involving fractional Laplacians and critical singularities. \textit{Adv. Nonlinear Stud.} 15 (2015), no. 3, 527--555.
 
 \bibitem{JW}{\sc S. Jarohs, T. Weth.} Asymptotic symmetry for a class of nonlinear fractional reaction-
diffusion equations. \textit{Discrete Contin. Dyn. Syst.} 34 (2014), 2581--2615.

\bibitem{JLX}{\sc  T. Jin; Y. Y. Li; J. Xiong}, On a fractional Nirenberg problem, part I: blow up analysis and compactness of solutions. \textit{J. Eur. Math. Soc.} (JEMS) 16 (2014), no. 6, 1111--1171.

\bibitem{MMPT}{\sc M. K. Kwong; J . B. Mcleod;  L. A. Peletier;  W. C. Troy}, On ground state
solutions of $-\De u = u^p - u^q$, \textit{J. Differential Equations}, 95, (1992), 218--239.

\bibitem{MP1} {\sc F. Merle; L. Peletier}, Asymptotic behaviour of positive solutions of elliptic equations with critical and supercritical growth. I. The radial case. \textit{ Arch. Rational Mech. Anal.} 112 (1990), no. 1, 1--19.

\bibitem{MP2} {\sc F. Merle; L. Peletier,} Asymptotic behaviour of positive solutions of elliptic equations with critical and supercritical growth. II. The non-radial case. \textit{J. Funct. Anal.} 105 (1992), no. 1, 1--41.

 \bibitem{Muckenhoupt}
{\sc B. Muckenhoupt}, Weighted norm inequalities for the Hardy maximal function,
\textit{ Trans. Amer. Math. Soc.} 165(1972), 207--226.

\bibitem{NPV}{\sc E. Di Nezza, G. Palatucci and E. Valdinoci}, Hitchhiker's guide to the fractional Sobolev spaces, \textit{Bull. Sci. Math.} 136 (2012), no. 5, 521--573.

\bibitem{PP}{\sc G. Palatucci; A. Pisante},  Improved Sobolev embeddings, profile decomposition, and concentration-compactness for fractional Sobolev spaces. \textit{Calc. Var. Partial Differential Equations} 50 (2014), no. 3-4, 799--829.

\bibitem{P}{\sc Y. J. Park}, Fractional Polya-Szego inequality. \textit{J. Chungcheong Math. Soc.} 24 (2011), no. 2,
267--271.

\bibitem{RS1} {\sc X. Ros-Oton; J. Serra},  Regularity theory for general stable operators. \textit{J. Differential Equations} 260 (2016), no. 12, 8675--8715.

\bibitem{RS4}{\sc  X. Ros-Oton; J. Serra},  The Dirichlet problem for the fractional Laplacian: regularity up to the boundary. \textit{J. Math. Pures Appl.} (9) 101 (2014), no. 3, 275--302.

\bibitem{RS2}{\sc X. Ros-Oton; J. Serra},  The Pohozaev identity for the fractional Laplacian. \textit{Arch. Ration. Mech. Anal.} 213 (2014), no. 2, 587--628.

\bibitem{RS3}{\sc  X. Ros-Oton; J. Serra}, Nonexistence results for nonlocal equations with critical and supercritical nonlinearities. \textit{Comm. Partial Differential Equations} 40 (2015), no. 1, 115--133.

\bibitem{SV3}{\sc R. Servadei; E. Valdinoci}, Weak and viscosity solutions of the fractional Laplace equation. \textit{Publ. Mat.} 58 (2014), no. 1, 133--154.

\bibitem{SerVal2}
{\sc R. Servadei; E. Valdinoci,} The Brezis-Nirenberg result for the fractional Laplacian.
\textit{  Trans. Amer. Math. Soc.} 367 (2015), no. 1, 67--102.

\bibitem{SerVal3}
{\sc R. Servadei; r. Valdinoci}, Mountain pass solutions for non-local elliptic operators. 
\textit{ J. Math. Anal. Appl.} 389 (2012), no. 2, 887--898.


\bibitem{TX} {\sc J. Tan; J. Xiong}, A Harnack inequality for fractional Laplace equations with lower order terms. \textit{Discrete Contin. Dyn. Syst.} 31 (2011), no. 3, 975--983.

 \bibitem{V}{\sc E. Valdinoci}, From the long jump random walk to the fractional Laplacian, \textit{Bol. Soc. Esp. Mat. Apl. SeMA No.} 49 (2009), 33--44.  
 
 \bibitem{VIKH}{\sc L. Vlahos; H. Isliker; K.  Kominis; K. Hizonidis}, \textit{Normal and anomalous diffusion : a tutorial}, in: T.Bountis(Ed.),Order and Chaos,10thVolume, Patras University Press, 2008.   


\end{thebibliography}
\end{document}